\documentclass[11 pt]{amsart}
\usepackage{graphicx}
\usepackage[hidelinks]{hyperref}
\usepackage{color}
\usepackage{amssymb,amsmath,amsfonts,latexsym}
\usepackage{epstopdf}
\usepackage{nicefrac}
\usepackage{amsthm}
\usepackage{hyperref}
\hypersetup{
     colorlinks   = true,
     citecolor    = blue
}

\usepackage{tikz}

\usepackage{tikz,fullpage}
\usetikzlibrary{patterns}
\usepackage{enumerate}
\usepackage{ulem}
\usepackage{float, lscape}
\numberwithin{equation}{section}
\usepackage[margin=1.27in]{geometry}
\usetikzlibrary{arrows,%
                petri,%
                topaths}%
\usepackage{tkz-berge}
\usepackage[position=top]{subfig}
\usetikzlibrary{%
  matrix,%
  calc,%
  arrows%
}
\newtheorem{thm}{THEOREM}[section]
\newtheorem{conj}[thm]{CONJECTURE}

\newtheorem{lemma}[thm]{LEMMA}

\newcommand{\G}{\Gamma}

\begin{document}
\title{Cyclic Cohomology and Chern Connes pairing of some crossed product algebras}
\author{Safdar Quddus}

\date{\today}
 
\let\thefootnote\relax

\begin{abstract}
We compute the cyclic and Hochschild cohomology groups for the algebras $\mathcal A_\theta^{alg} \rtimes \mathbb Z_3, \mathcal A_\theta^{alg} \rtimes \mathbb Z_4$ and $\mathcal A_\theta^{alg} \rtimes \mathbb Z_6$. We also compute the partial Chern-Connes index table for each of these algebras.
\end{abstract}
\maketitle
\section{Introduction}
The homological properties of noncommutative algebras have interested several mathematicians in recent years. In the articles \cite{C} and \cite{Y}, the (co)homology groups of smooth algebras having a C*-algebra structure are studied. The classical noncommutative algebras have been studied by Alev and Lambre \cite{AL}, Baudry\cite{B}, Fryer \cite{F}, Berest et al.\cite{BRT} and Quddus (\cite{Q1} and \cite{Q2}). \par

For given $\theta \notin \mathbb Q$, we associate the algebraic noncommutative torus as the algebra $\mathcal A_\theta$ defined to be
$$\mathcal A_\theta^{alg} :=\left\{a=\displaystyle\sum_{(n,m)\in \mathbb Z^2} a_{n,m} U_1^n U_2^m \mid a_{n,m} =0 \text{ for all but finitely many } (n,m) \right\},$$
where $U_1$ and $U_2$ are unitary generators satisfying $U_2 U_1 = \lambda U_1 U_2$, $\lambda =e^{2 \pi i \theta}$. 

The group $SL(2,\mathbb Z)$ has the following action on $\mathcal A_\theta^{alg}$. An element 
$$g= \left[
 \begin{array}{cc}
   g_{1,1} & g_{1,2} \\
   g_{2,1} & g_{2,2}
 \end{array} \right]\in SL(2,\mathbb Z)
$$ acts on the generators $U_1$ and $U_2$ as described below:
$$g \cdot U_1=e^{(\pi i g_{1,1} g_{2,1})\theta}U_1^{g_{1,1}}U_2^{g_{2,1}} \text{ and } g \cdot U_2=e^{(\pi i g_{1,2} g_{2,2})\theta}U_1^{g_{1,2}}U_2^{g_{2,2}}.$$
The algebra $\mathcal A_\theta^{alg}$ and associated algebras has been studied in several articles. While the authors of \cite{B} and \cite{O} computed some of the Hochschild homology groups of its $\mathbb Z_2, \mathbb Z_3, \mathbb Z_4 \text{ and } \mathbb Z_6$ crossed products, these groups were completely known in \cite{Q1}. In the paper \cite{BRT} the authors calculated the Picard group and the Morita equivalence classes of $\mathcal A_\theta^{alg}$.

In this article we compute the Hochschild and cyclic cohomology groups of the crossed product algebras $\mathcal A_\theta^{alg} \rtimes \mathbb Z_3, \mathcal A_\theta^{alg} \rtimes \mathbb Z_4$ and $\mathcal A_\theta^{alg} \rtimes \mathbb Z_6$ and thereafter we present the Chern-Connes index table by pairing the known projections of each of these algebras with the cyclic cocycles calculated in this article.
This article hence completes the Hochschild and cyclic cohomological description of the crossed product algebras obtained by the action of the discrete subgroups of $SL(2, \mathbb Z)$ on the algebraic noncommutative torus algebra. In this article we confine ourselves to the notations used in \cite{Q1} and \cite{Q2}.
\section{Statements}
The following are the statements of the theorems proved in this article.
\begin{thm} \label{thm:hoch}
 The Hochschild cohomology of the crossed product algebras are as follows:
 \begin{center}
$H^0(\mathcal A_{\theta}^{alg} \rtimes \G,(\mathcal A_{\theta}^{alg} \rtimes \G)^{\ast} ) \cong\begin{cases}
\mathbb C^7 & \text{ for } \G = \mathbb Z_3\\
\mathbb C^8  & \text{ for } \G = \mathbb Z_4 \\
\mathbb C^9 & \text{ for } \G = \mathbb Z_6. \end{cases}$\\
$H^1(\mathcal A_{\theta}^{alg} \rtimes \G,(\mathcal A_{\theta}^{alg} \rtimes \G)^{\ast}) \cong 0 \text{ for the finite subgroups } \G = \mathbb Z_3, \mathbb Z_4$ and $\mathbb Z_6$. \\
$H^2(\mathcal A_{\theta}^{alg} \rtimes \G,(\mathcal A_{\theta}^{alg} \rtimes \G)^{\ast})  \cong \mathbb C \text{ for the finite subgroups } \G = \mathbb Z_3, \mathbb Z_4$ and $\mathbb Z_6$. \\
$H^k(\mathcal A_{\theta}^{alg} \rtimes \G,(\mathcal A_{\theta}^{alg} \rtimes \G)^{\ast})  \cong 0 \text{ for all } k>3 \text{ and the finite subgroups } \G = \mathbb Z_3, \mathbb Z_4$ and $\mathbb Z_6$. \\
\end{center}
\end{thm}
\begin{thm} \label{thm:cyclic} The periodic cyclic cohomology groups are as follows:
\begin{center}
$HP^{even}(\mathcal A_{\theta}^{alg} \rtimes \G) \cong\begin{cases}
\mathbb C^8 & \text{ for } \G = \mathbb Z_3\\
\mathbb C^9  & \text{ for } \G = \mathbb Z_4 \\
\mathbb C^{10} & \text{ for } \G = \mathbb Z_6. \end{cases}$\\
$HP^{odd}(\mathcal A_{\theta}^{alg} \rtimes \G) \cong 0 \text{ for the finite subgroups } \G = \mathbb Z_3, \mathbb Z_4$ and $\mathbb Z_6$. \\
\end{center}
\end{thm}

\begin{thm} \label{thm:table} Let $\zeta = e^{\frac{2 \pi i }{6}}$, the following are the Chern-Connes index tables for the respective crossed product algebras. 
 \begin{center}
 (a) For $\mathcal A_\theta^{alg} \rtimes \mathbb Z_3$ : \space
  \begin{tabular}{c || c | c | c | c | c | c | c | c ||}
    
     &   S$\tau$ & $S\mathcal E_{0,0}^\omega$ & $S\mathcal E_{0,0}^{\omega^2}$ & $S\mathcal E_{0,1}^\omega$ & $S\mathcal E_{0,1}^{\omega^2}$ & $S\mathcal E_{0,-1}^{\omega}$ & $S\mathcal E_{0,-1}^{\omega^2}$ & $S\varphi$ \\ \hline \hline
    $1$ & $1$ & $0$ & $0$ & $0$ & $0$ & $0$ & $0$ & $0$\\ \hline
    $p_0^\theta$ & $\frac{1}{3}$ & $\frac{1}{3}$ & $\frac{1}{3}$ & $0$ & $0$ & $0 $& $0$ & $0$\\ \hline
    $p_1^\theta$ & $\frac{1}{3}$ & $\frac{\zeta^2}{3}$ & $\frac{\zeta^4}{3}$ &$ 0$ & $0$ & $0$ & $0$ & $0$\\ \hline
   $q_0^\theta$ & $\frac{1}{3}$ & $0$ & $0$ & $0$ & $\frac{\zeta^2}{3\sqrt[3]{\lambda^2}}$ & $0$ &$ 0$ & $0$\\ \hline
    $q_1^\theta$ & $\frac{1}{3}$ & $0$ & $0$ & $0$ &$\frac{1}{3\sqrt[3]{\lambda^2}}$ & $0$ & $0$ & $0$\\ \hline
   $r_0^\theta$ & $\frac{1}{3}$ & $0$ & $0$ & $0$ & $0$ & $\frac{\zeta^2 \sqrt[6]{\lambda}}{3}$ & $0$ & $0$\\ \hline
    $r_1^\theta$ & $\frac{1}{3}$ & $0$ & $0$ & $0$ & $0$ & $\frac{\zeta^4 \sqrt[6]{\lambda}}{3}$ & $0$ & $0$\\ \hline
    \hline
  \end{tabular}
\end{center}
\begin{center}
(b) For $\mathcal A_\theta^{alg} \rtimes \mathbb Z_4$ : \space
  \begin{tabular}{c || c | c | c | c | c | c | c | c | c ||} 
    
     &  S$\tau$ & $S\mathcal D_{1,1}$ & $S\mathcal D_{0,0}$ & $S(\mathcal D_{0,1}+\mathcal D_{1,0})$ & $S\mathcal F^i_{0,0}$ &  $S\mathcal F^i_{0,1}$ & $S\mathcal F^{-i}_{0,0}$ & $S\mathcal F^{-i}_{0,1}$ &$\varphi$ \\ \hline \hline
    1 & $1$ & $0$ & $0$ & $0$ & $0$ & $0$ & $0$ & $0$ & $0$\\ \hline
    $p_0^\theta$ & $\frac{1}{4}$ & $0$ & $\frac{1}{4}$ & $0$ & $\frac{1}{4}$ & $0$ & $\frac{1}{4}$ & $0$ & $0$\\ \hline
    $p_1^\theta$ & $\frac{1}{4}$ & $0$ & -$\frac{1}{4}$ & $0$ & $\frac{i}{4}$ & $0$ & $\frac{-i}{4}$ & $0$ & $0$\\ \hline
   $p_2^\theta$ & $\frac{1}{4}$ & $0$ & -$\frac{1}{4}$ & $0$ & -$\frac{1}{4}$ & $0$ & $-\frac{1}{4}$ & $0$ & $0$\\ \hline
    $q_0^\theta$ & $\frac{1}{4}$ & $-\frac{1}{4\sqrt\lambda}$ & $0$ & $0$ & $0$ & $\frac{i  \sqrt[4]{\lambda}}{4}$ & $0$ & $-\frac{i}{4 \sqrt[4]{\lambda}}$ & $0$\\ \hline
     $q_1^\theta$ & $\frac{1}{4}$ & $\frac{1}{4\sqrt\lambda}$ & $0$ & $0$ & $0$ & $\frac{- \sqrt[4]{\lambda}}{4}$ & $0$ & $-\frac{1}{4 \sqrt[4]{\lambda}}$ & $0$\\ \hline
     $q_2^\theta$ & $\frac{1}{4}$ & $-\frac{1}{4\sqrt\lambda}$ & $0$ & $0$ & $0$ & $\frac{-i\sqrt[4]{\lambda}}{4}$ & $0$ & $\frac{i}{4 \sqrt[4]{\lambda}}$ & $0$\\ \hline
       $r^\theta$ & $\frac{1}{2}$ & $0$ & $0$ & -$\frac{1}{2}$ & $0$ & $0$ & $0$ & $0$ & $0$\\ \hline
    \hline
  \end{tabular}
\end{center}
\begin{landscape}  
\begin{center}
(c) For $\mathcal A_\theta^{alg} \rtimes \mathbb Z_6$ : \space

\begin{tabular}{c || c | c | c | c | c | c | c | c | c | c ||}

     &  S$\tau$ &  $S\mathcal D_{0,0}$ &  $S(\mathcal D_{1,0}+\lambda \sqrt{\lambda}\mathcal D_{1,0}+\sqrt\lambda \mathcal D_{1,1})$& $S(\mathcal E^\omega_{0,1}+\mathcal E^\omega_{0,-1})$&  $S\mathcal E^{\omega}_{0,0}$ & $S(\mathcal E^{\omega^2}_{0,1}+\mathcal E^{\omega^2}_{0,-1})$ &  $S\mathcal E^{\omega^2}_{0,0}$& $S\mathcal G^{-\omega}_{0,0}$ & $S\mathcal G^{-\omega^2}_{0,0}$ & $\varphi$  \\ \hline \hline
    1 & $1$ & $0$ & $0$ & $0$ & $0$ & $0$ & $0$ & $0$ & $0$ & $0$\\ \hline
    $p_0^\theta$ & $\frac{1}{6}$ & $\frac{1}{6}$ & $0$ & $0$ & $\frac{1}{6}$ & $0$ & $\frac{1}{6}$ & $\frac{1}{6}$ & $\frac{1}{6}$ & $0$\\ \hline
    $p_1^\theta$ & $\frac{1}{6}$ & $\frac{1}{6}$ & $0$ & $0$ & $\frac{\zeta^2}{3}$ & $0$ & $-\frac{\zeta}{6}$ & $\frac{\zeta}{6}$ & $-\frac{\zeta^2}{6}$ & $0$\\ \hline
   $p_2^\theta$ & $\frac{1}{6}$ & $-\frac{1}{6}$ & $0$ & $0$ & $-\frac{\zeta}{3}$ & $0$ & -$\frac{1}{6}$ & $\frac{\zeta^2}{6}$ & -$\frac{\zeta}{6}$ & $0$\\ \hline
    $p_3^\theta$ & $\frac{1}{6}$ & $\frac{1}{6}$ & $0$ & $0$ & $\frac{1}{6}$ & $0$ & $\frac{1}{6}$ & -$\frac{1}{6}$ & -$\frac{1}{6}$ & $0$\\ \hline
    $p_4^\theta$ & $\frac{1}{6}$ & $0$ & $0$ & $0$ & $\frac{\zeta^2}{6}$ & $0$ & -$\frac{\zeta}{6}$ & -$\frac{\zeta}{6}$ & $\frac{\zeta^2}{6}$ & $0$\\ \hline
    $q_0^\theta$ & $\frac{1}{3}$ & $0$ & $0$ & $\frac{\zeta}{3}$ & $0$ & $\frac{\zeta^2}{3 \sqrt[6]{\lambda}}$ & $0$ & $0$ & $0$ & $0$\\ \hline
    $q_1^\theta$ & $\frac{1}{3}$ & $0$ & $0$ & $-\frac{1}{3}$ & $0$ & $-\frac{\zeta}{3 \sqrt[6]{\lambda}}$  & $0$ & $0$ & $0$ & $0$\\ \hline
    $r^\theta$ & $\frac{1}{2}$ & $0$ & $-\frac{\lambda\sqrt\lambda}{2}$ & $0$ & $0$ & $0$ & $0$ & $0$ & $0$ & $0$\\ \hline
    \hline
  \end{tabular}
\end{center}
\end{landscape}
\end{thm}

\section{Hochschild Cohomology}
We note that the dual of the algebraic noncommutative torus $\mathcal A_\theta^{alg}$ is 
$$\mathcal A_\theta^{alg \ast} =\left\{a \mid a =\displaystyle\sum_{(n,m)\in \mathbb Z^2} a_{n,m} U_1^n U_2^m\right\}.$$
For $g \in \Gamma$, ${}_{g}\mathcal A_\theta^{alg \ast}$ is the $g$-twisted $\mathcal A_\theta^{alg}$ bimodule, it consists of elements of $\mathcal A_\theta^{alg \ast}$ with the following twisted $\mathcal A_\theta^{alg}$ 
bimodule structure. For $a \in {}_{g}\mathcal A_\theta^{alg \ast}$ and $ \alpha \in \mathcal A_\theta^{alg}$,
$$\alpha \cdot a = (g \cdot \alpha)a\text{ and }a \cdot \alpha = a\alpha.$$ 
We outline the procedure to calculate the Hochschild cohomology groups, using the paracyclic decomposition technique \cite[Proposition~4.6]{GJ} we have the following decomposition  
$$H^\bullet(\mathcal A_\theta^{alg} \rtimes \Gamma, (\mathcal A_\theta^{alg} \rtimes \Gamma)^{\ast}) =  \displaystyle \bigoplus_{g \in \Gamma} H^\bullet(\mathcal A_\theta^{alg}, {}_{g}\mathcal A_\theta^{alg \ast})^{\Gamma}.$$
Using the well-known bimodule resolution 
$$ 0 \rightarrow (\mathcal A_\theta^{alg})^e \rightarrow (\mathcal A_\theta^{alg})^e \oplus (\mathcal A_\theta^{alg})^e \rightarrow (\mathcal A_\theta^{alg})^e$$
for $\mathcal A_\theta^{alg}$ with maps $ 1 \mapsto (U_2 \otimes 1 - \lambda U_2 \otimes 1 , 1 \otimes U_1-1 \otimes U_1)$ and $(1,0) \mapsto (U_1 \otimes 1-1\otimes U_1)$ and $(0,1) \mapsto (U_2 \otimes 1-1\otimes U_2)$. Hence for any bimodule $M$ of $\mathcal A_\theta^{alg}$, $H^\bullet(\mathcal A_\theta^{alg}, M)$ is computed from the following complex
$$ M \rightarrow M \oplus M \rightarrow M \rightarrow 0$$
in which the maps are $m \mapsto (U_1 m -m U_1, U_2 m -m U_2)$ and $(m_1,m_2) \mapsto (U_2 m_1-\lambda m_1 U_2-\lambda U_1 m_2 -m_2 U_1)$. To calculate the group $H^\bullet(\mathcal A_\theta^{alg}, {}_{g}\mathcal A_\theta^{alg \ast})$ we use the modified Connes resolution for the algebra $\mathcal A_\theta^{alg}$. We refer \cite[page 326]{Q1} for detailed discussion
on the modified Connes resolution for the algebra $\mathcal A_\theta^{alg}$. To locate the $\Gamma$ invariant $\bullet$-cocycles in $H^\bullet(\mathcal A_\theta^{alg}, {}_{g}\mathcal A_\theta^{alg \ast})$ we push 
a cocycle into the bar Hochschild cohomology complex and after the $\G$ action on it, we pull it back onto the Connes complex using the chain homotopy maps calculated explicitly in \cite[page 329]{Q1} and \cite[page 134]{C}. A comparison between the two $\bullet$-cocycles on the 
Connes complex will determine invariance. \par

\subsection{The Hochschild cohomology groups $H^0(\mathcal A_\theta^{alg} \rtimes \Gamma, (\mathcal A_\theta^{alg} \rtimes \Gamma)^{\ast}))$}~\\

\centerline{\uline{The case $\Gamma = \mathbb Z_3$.}}

The group $\mathbb Z_3$ is embedded in $SL(2,\mathbb Z)$ through its generator $\omega= \left[
 \begin{array}{cc}
   0 & 1 \\
   -1 & -1
 \end{array} \right]\in SL(2,\mathbb Z)$. The generator acts on $\mathcal A_\theta^{alg}$ in the following way
\begin{center}
$U_1 \mapsto U_2^{-1}, U_2 \mapsto \displaystyle \frac{U_1 U_2^{-1}}{\sqrt\lambda}$.
\end{center}

The case $g=1$ has been considered in the article \cite{Q2}, wherein the author calculated the Hochschild cohomology group $H^0(\mathcal A_\theta^{alg}, \mathcal A_\theta^{alg \ast})$. The said group is a one-dimensional group generated by the cocycle $\tau$ \cite[Section 2]{Q2}.
To check the $\mathbb Z_3$ invariance of $\tau$ we use the method described above. We push the 0-cocycle $\tau$ to the bar complex and pull it back to the Connes complex after the action of the group $\mathbb Z_3$. We notice that in this case the cocycle $\tau$ 
is invariant under the $\mathbb Z_3$ action and hence $H^0(\mathcal A_\theta^{alg}, \mathcal A_\theta^{alg \ast})^{\mathbb Z_3} \cong \mathbb C$. It is evident that $H^0(\mathcal A_\theta^{alg}, \mathcal A_\theta^{alg \ast})^{\mathbb Z_4} \cong H^0(\mathcal A_\theta^{alg}, \mathcal A_\theta^{alg \ast})^{\mathbb Z_6} \cong \mathbb C$.\par

The cases $g=\omega$ and $g=\omega^2$ are similar and we shall consider one of them for computational purpose(say $g=\omega$). We notice that following is the Hochschild cohomology complex for $g=\omega$ 
$${}_{\omega}\mathcal A_\theta^{alg \ast} \xrightarrow{{}_{\omega}\alpha_1}{}_{\omega}\mathcal A_\theta^{alg \ast} \oplus {}_{\omega}\mathcal A_\theta^{alg \ast} \xrightarrow{{}_{\omega}\alpha_2} {}_{\omega}\mathcal A_\theta^{alg \ast} \rightarrow 0,$$
wherein with the bimodule structure of ${}_{\omega}\mathcal A_\theta^{alg \ast}$, the maps are as follows:

$${}_{\omega}\alpha_1(\varphi) = (U_2^{-1}\varphi - \varphi U_1,\displaystyle \frac{U_1U_2^{-1}}{\sqrt{\lambda}} \varphi - \varphi U_2)\; ;{}_{\omega}\alpha_2(\varphi_1, \varphi_2) = \displaystyle \frac{U_1U_2^{-1}}{\sqrt{\lambda}}\varphi_1-\lambda \varphi_1 U_2 - \lambda U_2^{-1} \varphi_2 + \varphi_2 U_1.$$

Hence the $\omega$ twisted Hochschild cohomology group $H^0(\mathcal A_\theta^{alg}, {}_{\omega}\mathcal A_\theta^{alg \ast})$ is the group $ker({}_{\omega}\alpha_1)$.
The group $ker({}_{\omega}\alpha_1)$ is the set of all elements $\varphi \in \mathcal A_\theta^{alg \ast}$ such that the following relations are satisfied:
\begin{center}
 $U_2^{-1} \varphi = \varphi U_1$ and $\displaystyle\frac{U_1 U_2^{-1}}{\sqrt{\lambda}} \varphi = \varphi U_2$.
\end{center}
Hence, we deduce that for $\varphi$ an element of $ker({}_{\omega}\alpha_1)$, its coefficients must satisfy the following:
$$\varphi_{n,{m+1}} = \lambda^{m+n} \varphi_{n-1,m}\text{ and }\varphi_{n-1,{m+1}} = \lambda^{n- \frac{1}{2}} \varphi_{n,m-1}.$$
\begin{lemma} The $\omega$ twisted zeroth Hochschild cohomology group $H^0(\mathcal A_\theta^{alg}, {}_{\omega}\mathcal A_\theta^{alg \ast})$  is generated by the coefficients $\varphi_{0,0}$(generates the cocycle $\mathcal E^\omega_{0,0}$), $\varphi_{0,1}$(generates the cocycle $\mathcal E^\omega_{0,1}$) and $\varphi_{0,-1}$(generates the cocycle $\mathcal E^\omega_{0,-1}$). And satisfies the following relations:\newline 
For $m -n \equiv 0 \pmod{3} $, $\varphi_{n,m} = \lambda^\frac{m^2+n^2+4mn}{6} \varphi_{0,0}$ and \\
for $m-n \equiv \pm 1 \pmod{3}$, $\varphi_{n,m} = \lambda^\frac{m^2+n^2+4mn-1}{6} \varphi_{0,\pm 1}$.
 \end{lemma}
\begin{proof}
Using the relation $\varphi_{n,m+1} = \lambda^{n+m} \varphi_{n-1,m}$  we infer that the lattice points $(n,m)$ and $(n-1,m-1)$ belong to the same coboundary hence represent the same cocycle element. Furthermore the second relation $\varphi_{n-1,{m+1}} = \lambda^{n- \frac{1}{2}} \varphi_{n,m-1}$ relates $\varphi_{0,k}$ with $\varphi_{-1,{k+2}}$. Thus, we conclude that the group $H^0(\mathcal A_\theta^{alg}, {}_{\omega}\mathcal A_\theta^{alg \ast})$ is generated by the coefficients $\varphi_{0,0}$, $\varphi_{0,1}$ and $\varphi_{0,-1}$. This can be pictorially understood through the  following diagram.

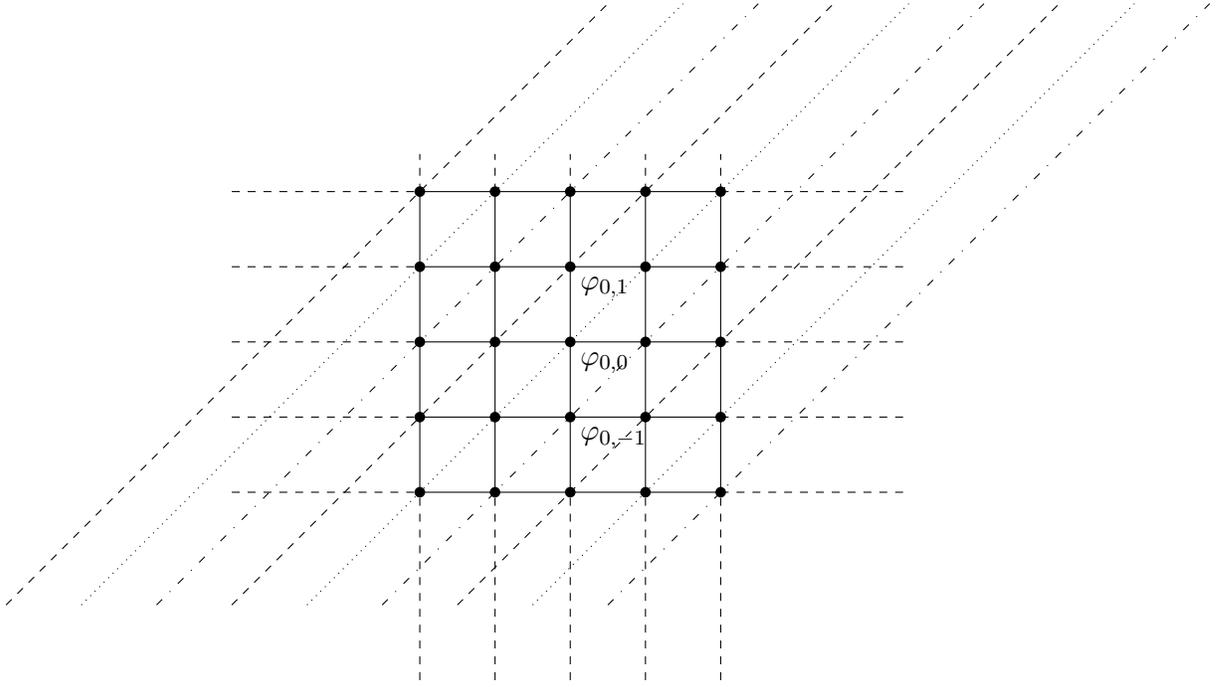
\begin{center}
\begin{tikzpicture}
\draw[step=1.0, black, thin, xshift=.5cm, yshift=.5cm](-2,-2) grid(2,2);
\fill (.5,0.5) circle (2pt)node[below right]{$\varphi_{0,0}$};
\fill (-.5,0.5) circle (2pt);
\fill (1.5,0.5) circle (2pt);
\fill (-1.5,0.5) circle (2pt);
\fill (2.5,0.5) circle (2pt);
\fill (.5,1.5) circle (2pt)node[below right]{$\varphi_{0,1}$};
\fill (-.5,1.5) circle (2pt);
\fill (1.5,1.5) circle (2pt);
\fill (-1.5,1.5) circle (2pt);
\fill (2.5,1.5) circle (2pt);
\fill (.5,-.5) circle (2pt)node[below right]{$\varphi_{0,-1}$};
\fill (-.5,-.5) circle (2pt);
\fill (1.5,-.5) circle (2pt);
\fill (-1.5,-.5) circle (2pt);
\fill (2.5,-.5) circle (2pt);
\fill (.5,2.5) circle (2pt);
\fill (-.5,2.5) circle (2pt);
\fill (1.5,2.5) circle (2pt);
\fill (-1.5,2.5) circle (2pt);
\fill (2.5,2.5) circle (2pt);
\fill (.5,-1.5) circle (2pt);
\fill (-.5,-1.5) circle (2pt);
\fill (1.5,-1.5) circle (2pt);
\fill (-1.5,-1.5) circle (2pt);
\fill (2.5,-1.5) circle (2pt);
\draw[dashed](-4,-3)--(4,5)node[right]{};
\draw[dotted](-3,-3)--(5,5)node[right]{};
\draw[loosely dashdotted](-2,-3)--(6,5)node[right]{};
\draw[dashed](-1,-3)--(7,5)node[right]{};
\draw[dotted](0,-3)--(8,5)node[right]{};
\draw[loosely dashdotted](1,-3)--(9,5)node[right]{};
\draw[dashed](-7,-3)--(1,5)node[right]{};
\draw[dotted](-6,-3)--(2,5)node[right]{};
\draw[loosely dashdotted](-5,-3)--(3,5)node[right]{};
\draw[dashed](-1.5,2.5)--(-1.5,3)node[right]{};
\draw[dashed](-0.5,2.5)--(-0.5,3)node[right]{};
\draw[dashed](0.5,2.5)--(0.5,3)node[right]{};
\draw[dashed](1.5,2.5)--(1.5,3)node[right]{};
\draw[dashed](2.5,2.5)--(2.5,3)node[right]{};
\draw[dashed](-1.5,-4.0)--(-1.5,-1.5)node[right]{};
\draw[dashed](-0.5,-4.0)--(-0.5,-1.5)node[right]{};
\draw[dashed](0.5,-4.0)--(0.5,-1.5)node[right]{};
\draw[dashed](1.5,-4.0)--(1.5,-1.5)node[right]{};
\draw[dashed](2.5,-4.0)--(2.5,-1.5)node[right]{};
\draw[dashed](-4,1.5)--(-1.5,1.5)node[right]{};
\draw[dashed](-4,2.5)--(-1.5,2.5)node[right]{};
\draw[dashed](-4,-1.5)--(-1.5,-1.5)node[right]{};
\draw[dashed](-4,-.5)--(-1.5,-.5)node[right]{};
\draw[dashed](-4,.5)--(-1.5,.5)node[right]{};
\draw[dashed](2.5,1.5)--(5,1.5)node[right]{};
\draw[dashed](2.5,2.5)--(5,2.5)node[right]{};
\draw[dashed](2.5,-1.5)--(5,-1.5)node[right]{};
\draw[dashed](2.5,-.5)--(5,-.5)node[right]{};
\draw[dashed](2.5,.5)--(5,.5)node[right]{};

\end{tikzpicture}
\end{center}

Using the two relations, we easily see that for the case $m-n \equiv 0 \pmod{3}$, $\varphi_{0,3k} = \lambda^{\frac{3k^2}{2}}\varphi_{0,0}$. Further we notice that 
\begin{center}
$\varphi_{r, {3k+r}} = \lambda^{(3k+1)+(3k+3)+\cdots + (3k+2r-1)} \varphi_{0,3k}=\lambda^{(3rk+r^2)} \varphi_{0,3k}= \lambda^{(3rk+r^2)} \lambda^{\frac{3k^2}{2}}\varphi_{0,0}=\lambda^{\frac{(2r^2+6kr+3k^2)}{2}}\varphi_{0,0}$.
\end{center}
Hence, we have for $m -n \equiv 0 \pmod{3} $, $\varphi_{n,m} = \lambda^\frac{m^2+n^2+4mn}{6} \varphi_{0,0}$.
Now we consider the case $m -n \equiv 1 \pmod{3} $. Using the relations $\varphi_{n-1,{m+1}} = \lambda^{n- \frac{1}{2}} \varphi_{n,m-1}$ and $\varphi_{n,{m+1}} = \lambda^{m+n} \varphi_{n-1,m}$ successively we get that
$$\varphi_{0,3k+1} = \lambda^{\frac{3k^2}{2}+k}\varphi_{0,1}.$$
As in the previous case we notice that 
\begin{center}
$\varphi_{r, {3k+r+1}} = \lambda^{(3k+2)+(3k+4)+\cdots + (3k+2r)} \varphi_{0,3k+1}=\lambda^{(3rk+r(r+1))} \varphi_{0,3k+1}= \lambda^{(3rk+r(r+1))} \lambda^{\frac{3k^2}{2}+k}\varphi_{0,1}=\lambda^{\frac{(2r^2+6kr+3k^2+2r+2k)}{2}}\varphi_{0,1}$.
\end{center}
Hence, we have that for  $m -n \equiv 1 \pmod{3} $, $\varphi_{n,m} = \lambda^\frac{m^2+n^2+4mn-1}{6} \varphi_{0,1}$. A similar computation for the final case $m-n \equiv -1\pmod{3}$ completes the proof.
\end{proof}

\begin{lemma}
$H^0(\mathcal A_\theta^{alg} , {}_{\omega}\mathcal A_\theta^{alg \ast})^{\mathbb Z_3} \cong \mathbb C^3$.
\end{lemma}
\begin{proof}
The entry $\varphi_{n,m} U_1^n U_2^m$ under the action of $\omega$  transforms to $\varphi_{n,m} (U_2^{-n}) \displaystyle (\frac{U_1 U_2^{-1}}{\sqrt\lambda})^{m}$. Using the fact that $(U_1 U_2^{-1})^m = \lambda^{-\frac{m(m-1)}{2}} U_1^m U_2^{-m}$. We have 
\begin{center}
 $\varphi_{n,m} (U_2^{-n}) \displaystyle (\frac{U_1 U_2^{-1}}{\sqrt\lambda})^{m} = \varphi_{n,m}  \lambda^{\frac{-m}{2}}U_2^{-n} ({U_1 U_2^{-1}})^{m}=\varphi_{n,m} \lambda^{-\frac{m(m-1)}{2}} \lambda^{\frac{-m}{2}} U_2^{-n} U_1^m U_2^{-m} =  \varphi_{n,m} \lambda^{-\frac{m^2}{2}} \lambda^{-nm}  U_1^m U_2^{-n-m}$ .
\end{center}
We note that $((-n-m)-m)-(m-n) = -3m \equiv 0 \pmod{3}$, hence the coefficient of $\omega \cdot \varphi_{n,m} U_1^n U_2^m$ is generated by the same generator($\varphi_{0,0}$ or $\varphi_{0, \pm 1}$) which generates $\varphi_{n,m}$. \par

Let us consider the case $m-n \equiv 0 \pmod{3}$, in this case we have 
\begin{center}
$\varphi_{n,m} \lambda^{-\frac{m^2}{2}} \lambda^{-nm}  U_1^m U_2^{-n-m} = \lambda^{n^2+m^2+4nm} \lambda^{-\frac{m^2}{2}} \lambda^{-nm} \varphi_{0,0} U_1^m U_2^{-n-m}= \lambda^{\frac{n^2-2m^2-2nm}{6}}\varphi_{0,0} U_1^m U_2^{-n-m} = \lambda^{m^2+(-n-m)^2+4m(-n-m)} \varphi_{0,0} U_1^m U_2^{-n-m} = \varphi_{m, {-n-m}} U_1^m U_2^{-n-m}$
\end{center}
Hence we see that for $m-n \equiv 0 \pmod{3}$,  $\omega \cdot \varphi_{n,m} U_1^n U_2^m=\varphi_{m, {-n-m}} U_1^m U_2^{-n-m}$ and therefore is invariant under the action of $\omega$, the generator of the group $\mathbb Z_3$. For the cases $m-n \equiv \pm1 \pmod{3}$, the proof is similar and the key to the invariance is the polynomial $n^2+m^2+4mn$, which under the transformation $n \mapsto m$ and $m \mapsto {-n-m}$ satisfies the relation
$$n^2 + m^2 +4nm - 3m^2 - 6nm = m^2 + (-n-m)^2 + 4m(-n-m).$$
 \end{proof}

\begin{thm}  \label{thm:hoch03}
$H^0(\mathcal A_\theta^{alg} \rtimes \mathbb Z_3, (\mathcal A_\theta^{alg} \rtimes \mathbb Z_3)^ \ast) \cong \mathbb C^7$.
\end{thm}
\begin{proof} Using the paracyclic decomposition of the group $H^0(\mathcal A_\theta^{alg} \rtimes \mathbb Z_3, (\mathcal A_\theta^{alg} \rtimes \mathbb Z_3)^ \ast)$  we have the following relation
\begin{center}
$H^0(\mathcal A_\theta^{alg} \rtimes \mathbb Z_3, (\mathcal A_\theta^{alg} \rtimes \mathbb Z_3)^ \ast) = H^0(\mathcal A_\theta, {}_{\omega}\mathcal A_\theta^{alg \ast})^{\mathbb Z_3} \displaystyle \oplus  H^0(\mathcal A_\theta, {}_{\omega^2}\mathcal A_\theta^{alg \ast})^{\mathbb Z_3} \displaystyle \oplus  H^0(\mathcal A_\theta, \mathcal A_\theta^{alg \ast})^{\mathbb Z_3} = \mathbb C^3 \displaystyle \oplus \mathbb C^3 \displaystyle \oplus \mathbb C = \mathbb C^7$.
\end{center}
Hence proved.
\end{proof}

\centerline{\uline{The case $\Gamma = \mathbb Z_4$.}}
The group $\mathbb Z_4$ is embedded in $SL(2,\mathbb Z)$ through its generator $i= \left[
 \begin{array}{cc}
   0 & -1 \\
   1 & 0
 \end{array} \right]\in SL(2,\mathbb Z)$. The generator acts on $\mathcal A_\theta^{alg}$ in the following way
\begin{center}
$U_1 \mapsto U_2^{-1}, U_2 \mapsto U_1$.
\end{center}
Below is the paracyclic decomposition of the cohomology group $H^0(\mathcal A_\theta \rtimes \mathbb Z_4, (\mathcal A_\theta^{alg} \rtimes \mathbb Z_4)^{\ast})$;
$$H^0(\mathcal A_\theta^{alg} \rtimes \mathbb Z_4, (\mathcal A_\theta^{alg} \rtimes \mathbb Z_4)^{\ast}) = \displaystyle \bigoplus_{g \in \mathbb Z_4} H^0(\mathcal A_\theta^{alg}, {}_{g}\mathcal A_\theta^{alg \ast})^{\mathbb Z_4}.$$
We recall that the zeroth Hochschild cohomology group $H^0(\mathcal A_\theta^{alg}, \mathcal A_\theta^{alg \ast})$ is one dimensional and the generator $\tau$ is invariant under $\mathbb Z_4$ action. The $-1$ twisted zeroth cohomology group,  $H^\bullet(\mathcal A_\theta^{alg}, {}_{-1}\mathcal A_\theta^{alg})$ has been calculated in \cite{Q2}. For $g=\pm i$, we notice that the $i$ twisted Hochschild cohomology group  $H^\bullet(\mathcal A_\theta^{alg}, {}_{i}\mathcal A_\theta^{alg})$ and the $-i$ twisted Hochschild cohomology group $H^\bullet(\mathcal A_\theta^{alg}, {}_{-i}\mathcal A_\theta^{alg})$ are isomorphic. Hence we consider the case $g=i$ for our computational purpose. \par
For $g=i$ the following is the Hochschild cohomology complex for the Connes resolution.
\begin{center}
${}_{i}\mathcal A_\theta^{alg \ast} \xrightarrow{{}_{i}\alpha_1}{}_{i}\mathcal A_\theta^{alg \ast} \oplus {}_{i}\mathcal A_\theta^{alg \ast} \xrightarrow{{}_{i}\alpha_2} {}_{i}\mathcal A_\theta^{alg \ast} \rightarrow 0$,
\end{center}
where the cochain maps are as follows:
$${}_{i}\alpha_1(\varphi) = (U_2^{-1} \varphi - \varphi U_1, U_1 \varphi - \varphi U_2);\; {}_{i}\alpha_2(\varphi_1, \varphi_2) = U_1 \varphi_1 - \lambda \varphi_1 U_2 + \varphi_2 U_1 - \lambda U_2^{-1} \varphi_2.$$

Hence the $i$ twisted Hochschild cohomology group $H^0(\mathcal A_\theta^{alg}, {}_{i}\mathcal A_\theta^{alg \ast})$ is the group $ker({}_{i}\alpha_1)$.
The group $ker({}_{i}\alpha_1)$ is the set of all entries $\varphi \in \mathcal A_\theta^{alg \ast}$ such that the following relations hold:
\begin{center}
 $U_2^{-1} \varphi = \varphi U_1$ and $U_1 \varphi = \varphi U_2$.
\end{center}
Hence, we deduce that $\varphi_{n,{m+1}} = \lambda^{m+n} \varphi_{n-1,m}$ and $\varphi_{n-1,m} =  \varphi_{n,m-1}$.

\begin{lemma} The $i$ twisted zeroth Hochschild cohomology group $H^0(\mathcal A_\theta^{alg}, {}_{i}\mathcal A_\theta^{alg \ast})$ is generated by the coefficients $\varphi_{0,0}$(generates the cocycle $\mathcal F^i_{0,0}$) and $\varphi_{0,1}$(generates the cocycle $\mathcal F^i_{0,1}$). And satisfies the following relations:\\
For $m +n \equiv 0 \pmod{2} $, $\varphi_{n,m} = \lambda^\frac{m^2+n^2+2mn}{4} \varphi_{0,0}$ and \\
for $m+n \equiv 1 \pmod{2}$, $\varphi_{n,m} = \lambda^\frac{m^2+n^2+2mn-1}{4} \varphi_{0,1}$.
 \end{lemma}

\begin{proof}
The relation $\varphi_{n,m+1} = \lambda^{n+m} \varphi_{n-1,m}$ relates all the lattice points $(n,m)$ with $(n-1,m-1)$. Furthermore, using the relation $\varphi_{n-1,m} =  \varphi_{n,m-1}$ we see that  the lattice points $(n,m)$ with $(n-1,m+1)$ are generated by the same coefficient. It can be easily concluded that $\varphi_{0,0}$ generates all the coefficients $\varphi_{r,s}$ such that $r+s$ is even and the coefficient $\varphi_{0,1}$ generates all the coefficients $\varphi_{r,s}$ such that $r+s$ is odd. This can be pictorially realised by the following diagram

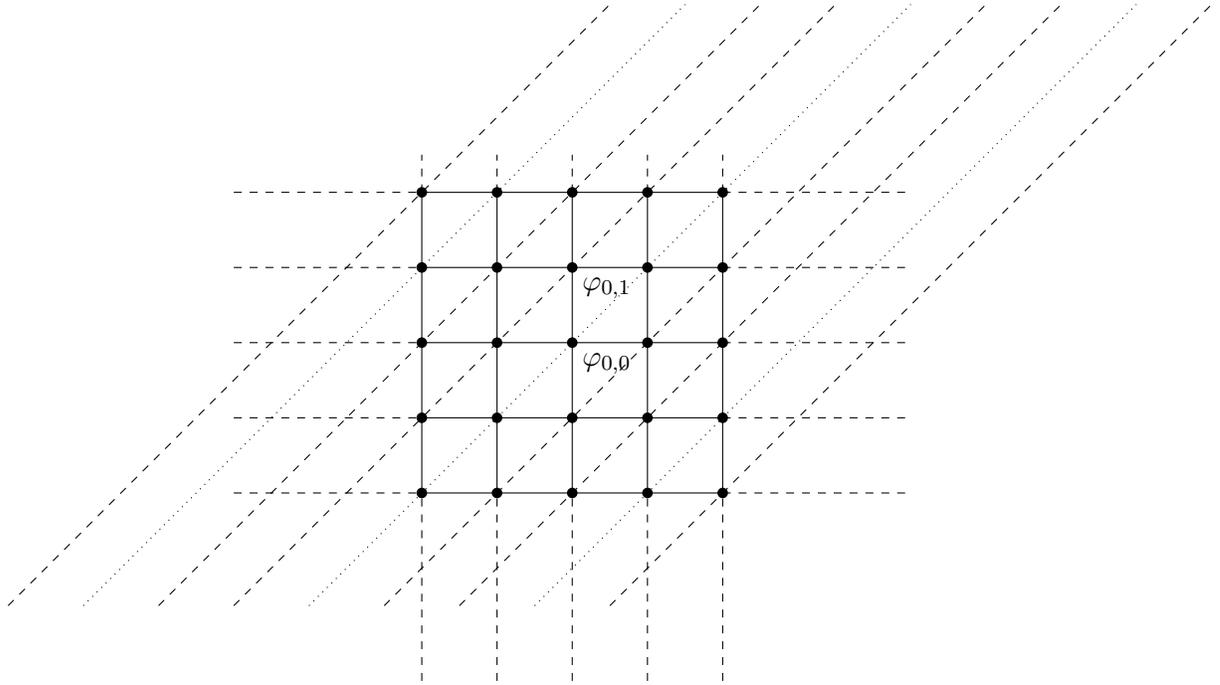
\begin{center}
\begin{tikzpicture}
\draw[step=1.0, black, thin, xshift=.5cm, yshift=.5cm](-2,-2) grid(2,2);
\fill (.5,0.5) circle (2pt)node[below right]{$\varphi_{0,0}$};
\fill (-.5,0.5) circle (2pt);
\fill (1.5,0.5) circle (2pt);
\fill (-1.5,0.5) circle (2pt);
\fill (2.5,0.5) circle (2pt);
\fill (.5,1.5) circle (2pt)node[below right]{$\varphi_{0,1}$};
\fill (-.5,1.5) circle (2pt);
\fill (1.5,1.5) circle (2pt);
\fill (-1.5,1.5) circle (2pt);
\fill (2.5,1.5) circle (2pt);
\fill (.5,-.5) circle (2pt);
\fill (-.5,-.5) circle (2pt);
\fill (1.5,-.5) circle (2pt);
\fill (-1.5,-.5) circle (2pt);
\fill (2.5,-.5) circle (2pt);
\fill (.5,2.5) circle (2pt);
\fill (-.5,2.5) circle (2pt);
\fill (1.5,2.5) circle (2pt);
\fill (-1.5,2.5) circle (2pt);
\fill (2.5,2.5) circle (2pt);
\fill (.5,-1.5) circle (2pt);
\fill (-.5,-1.5) circle (2pt);
\fill (1.5,-1.5) circle (2pt);
\fill (-1.5,-1.5) circle (2pt);
\fill (2.5,-1.5) circle (2pt);
\draw[dashed](-4,-3)--(4,5)node[right]{};
\draw[dotted](-3,-3)--(5,5)node[right]{};
\draw[dashed](-2,-3)--(6,5)node[right]{};
\draw[dashed](-1,-3)--(7,5)node[right]{};
\draw[dotted](0,-3)--(8,5)node[right]{};
\draw[dashed](1,-3)--(9,5)node[right]{};
\draw[dashed](-7,-3)--(1,5)node[right]{};
\draw[dotted](-6,-3)--(2,5)node[right]{};
\draw[dashed](-5,-3)--(3,5)node[right]{};
\draw[dashed](-1.5,2.5)--(-1.5,3)node[right]{};
\draw[dashed](-0.5,2.5)--(-0.5,3)node[right]{};
\draw[dashed](0.5,2.5)--(0.5,3)node[right]{};
\draw[dashed](1.5,2.5)--(1.5,3)node[right]{};
\draw[dashed](2.5,2.5)--(2.5,3)node[right]{};
\draw[dashed](-1.5,-4.0)--(-1.5,-1.5)node[right]{};
\draw[dashed](-0.5,-4.0)--(-0.5,-1.5)node[right]{};
\draw[dashed](0.5,-4.0)--(0.5,-1.5)node[right]{};
\draw[dashed](1.5,-4.0)--(1.5,-1.5)node[right]{};
\draw[dashed](2.5,-4.0)--(2.5,-1.5)node[right]{};
\draw[dashed](-4,1.5)--(-1.5,1.5)node[right]{};
\draw[dashed](-4,2.5)--(-1.5,2.5)node[right]{};
\draw[dashed](-4,-1.5)--(-1.5,-1.5)node[right]{};
\draw[dashed](-4,-.5)--(-1.5,-.5)node[right]{};
\draw[dashed](-4,.5)--(-1.5,.5)node[right]{};
\draw[dashed](2.5,1.5)--(5,1.5)node[right]{};
\draw[dashed](2.5,2.5)--(5,2.5)node[right]{};
\draw[dashed](2.5,-1.5)--(5,-1.5)node[right]{};
\draw[dashed](2.5,-.5)--(5,-.5)node[right]{};
\draw[dashed](2.5,.5)--(5,.5)node[right]{};

\end{tikzpicture}
\end{center}

Using the two relations, we easily see that $\varphi_{0,2k} = \lambda^{k^2}\varphi_{0,0}$
\begin{center}
$\varphi_{r, {2k+r}} = \lambda^{(2k+1)+(2k+3)+\cdots + (2k+(2r-1))} \varphi_{0,2k}=\lambda^{(2rk+r^2)} \varphi_{0,2k}= \lambda^{(2rk+r^2)} \lambda^{k^2}\varphi_{0,0}=\lambda^{(r+k)^2}\varphi_{0,0}$.
\end{center}
Hence we have the proved the theorem for case $m+n \equiv 0 \mod{2}$. A similar argument for the case $m+n \equiv 1 \pmod{2}$ completes the proof of the lemma.
\end{proof}
\begin{lemma}
$H^0(\mathcal A_\theta^{alg}, {}_{i}\mathcal A_\theta^{alg \ast})^{\mathbb Z_4} \cong \mathbb C^2$.
\end{lemma}
\begin{proof}
Equivalently, we need to show that the two cocycles $\mathcal F^i_{0,0}$ and $\mathcal F^i_{0,1}$ are $\mathbb Z_4$ invariant. The entry $\varphi_{n,m} U_1^n U_2^m$ under the action of $i$  transforms to $\varphi_{n,m} (U_2^{-n}) U_1^{m}$. Using the fact that $U_2^{-n} U_1^{m} = \lambda^{-nm} U_1^m U_2^{-n}$. We notice that if $m+n \equiv 0 \pmod{2}$ then $m-n \equiv 0 \pmod{2}$ and similar is the case for $m+n \equiv 1 \pmod{2}$. Hence it makes sense to compare the coefficient of $ i \cdot \varphi_{n,m} U_1^n U_2^m$ with $\varphi_{n,m}$. \par
Using the result of previous lemma, for $m+n \equiv 0 \pmod{2}$, we have 
\begin{center}
 $ \varphi_{n,m} U_2^{-n} U_1^{m}=\lambda^{-nm} \varphi_{n,m} U_1^m U_2^{-n} = \lambda^{-nm} \lambda^{\frac{n^2+m^2+2nm}{4}} \varphi_{0,0} U_1^m U_2^{-n}=\lambda^{\frac{n^2+m^2-2nm}{4}} \varphi_{0,0} U_1^m U_2^{-n}=\varphi_{m,-n} U_1^m U_2^{-n}$ .
\end{center}

Hence we see that for $m+n \equiv 0 \pmod{2}$,  $i \cdot \varphi_{n,m} U_1^n U_2^m=\varphi_{m,-n} U_1^m U_2^{-n}$ and therefore $i$ leaves the cocycle $\mathcal F^i_{0,0}$ invariant under its action. For the case $m+n \equiv 1 \pmod{2}$, the proof is similar and the key to the invariance is the polynomial $n^2+m^2+2mn$, which is under the transformation $n \mapsto m$ and $m \mapsto {-n}$ satisfies the relation
$$n^2 + m^2 +2nm - 4nm = m^2 + (-n)^2 + 2m(-n).$$
\end{proof}
\begin{lemma}
$H^0(\mathcal A_\theta^{alg}, {}_{-1}\mathcal A_\theta^{alg \ast})^{\mathbb Z_4} \cong \mathbb C^3$.

\end{lemma}
\begin{proof}
The cocycles of $H^0(\mathcal A_\theta^{alg}, {}_{-1}\mathcal A_\theta^{alg \ast})$ are described in \cite{Q2}. They are generated by the coefficients $\varphi_{0,0}$, $\varphi_{1,0}$, $\varphi_{0,1}$ and $\varphi_{1,1}$. A typical cocycle of $H^0(\mathcal A_\theta^{alg}, {}_{-1}\mathcal A_\theta^{alg \ast})$ is of the form 
\begin{center}
$\Phi= a \mathcal D_{0,0} + b\mathcal D_{0,1} + c\mathcal D_{1,0} + d\mathcal D_{1,1}$ where, $a,b,c,d \in \mathbb C$.
\end{center}
For $0 \leq i,j \leq 1$ the cocycles $\mathcal D_{i,j}$ are generated by the coefficients $\varphi_{i,j}$. Consider the cocycle $\mathcal D_{0,0}$, its entries are of the form $\varphi_{2n,2m}$. These are generated by $\varphi_{0,0}$ and satisfies the relation $\varphi_{2n,2m} = \lambda^{2nm} \varphi_{0,0}$. By the following calculations we conclude that the cocycle $\mathcal D_{0,0}$ is invariant under the action of $\mathbb Z_4$\begin{center}
 $i \cdot \varphi_{2n,2m} U_1^{2n} U_2^{2m} = \varphi_{2n,2m}U_2^{-2n}U_1^{2m}= \varphi_{2n,2m} \lambda^{-4nm} U_1^{2m} U_2^{2n}=\lambda^{2nm} \lambda^{-4nm} \varphi_{),0} U_1^{-2m} U_2^{2n} = \lambda^{-2nm} \varphi_{0,0} U_1^{-2m} U_2^{2n} = \varphi_{-2m,2n} U_1^{-2m} U_2^{2n}$.
\end{center}
Similarly we consider the cocycle $\mathcal D_{1,1}$, its entries satisfy the relation $\varphi_{2n+1, 2n+1} = \lambda^{2nm+n+m}\varphi_{1,1}$. We have the following relations
\begin{center} $i \cdot \varphi_{2n+1,2m+1} U_1^{2n+1} U_2^{2m+1} = \varphi_{2n+1,2m+1} U_2^{-2n-1}U_1^{2m+1}=\varphi_{2n+1,2m+1} \lambda^{-(2n+1)(2m+1)}U_1^{2m+1}U_2^{-2n-1}=\varphi_{1,1} \lambda^{2nm+n+m}\lambda^{-(2n+1)(2m+1)}U_1^{2m+1}U_2^{-2n-1}=\varphi_{1,1} \lambda^{2nm+n+m-4nm-2n-2m-1}U_1^{2m+1}U_2^{-2n-1}=\varphi_{1,1} \lambda^{-2nm-n-m-1}U_1^{2m+1}U_2^{-2n-1} =\varphi_{1,1} \lambda^{2(-n-1)m+m+(-n-1)}U_1^{2m+1}U_2^{2(-n-1)+1}=\varphi_{2m+1,-2n-1} U_1^{2m+1} U_2^{-2n-1} .$
\end{center}
Hence the cocycle $\mathcal D_{1,1}$ belongs to the group $H^0(\mathcal A_\theta^{alg}, {}_{-1}\mathcal A_\theta^{alg \ast})^{\mathbb Z_4}$. \par
We now observe the action of $i$ on the cocycle $\mathcal D_{0,1}$, an entry $\varphi_{2n, 2m+1} U_1^{2n} U_2^{2m+1}$ when acted by $i$ has the following transformations
\begin{center}
$i \cdot \varphi_{2n, 2m+1} U_1^{2n} U_2^{2m+1} = \varphi_{2n, 2m+1} U_2^{-2n} U_1^{2m+1} = \varphi_{2n, 2m+1}\lambda^{-2n(2m+1)} U_1^{2m+1} U_2^{-2n}=\varphi_{0,1} \lambda^{2nm+n}\lambda^{-2n(2m+1)} U_1^{2m+1} U_2^{-2n}=\varphi_{0,1} \lambda^{-2nm-n} U_1^{2m+1} U_2^{-2n} =\frac{\varphi_{0,1}}{\varphi_{1,0}}\varphi_{2m+1,-2n}U_1^{2m+1} U_2^{-2n}$. 
\end{center}
Similarly, when $i$ acts on the cocycle $\mathcal D_{1,0}$, the entries of the form $\varphi_{2n+1,2m} U_1^{2n+1} U_2^{2m}$ have the following transformations
\begin{center}
$i \cdot \varphi_{2n+1, 2m} U_1^{2n+1} U_2^{2m} = \varphi_{2n+1, 2m} U_2^{-2n-1} U_1^{2m} = \varphi_{2n+1, 2m}\lambda^{-2m(2n+1)} U_1^{2m} U_2^{-2n-1}=\varphi_{1,0} \lambda^{2nm+m}\lambda^{-2m(2n+1)} U_1^{2m} U_2^{-2n-1}=\varphi_{1,0} \lambda^{-2nm-m} U_1^{2m} U_2^{-2n-1} =\displaystyle \frac{\varphi_{1,0}}{\varphi_{0,1}}\varphi_{-2m,2n+1}U_1^{2m} U_2^{-2n-1}$. 
\end{center}
 Let us consider the action of $i$ on a typical cocycle $\Phi$. Then we have
\begin{center}
 $i \cdot \Phi = a( i\cdot \mathcal D_{0,0}) + b( i \cdot \mathcal D_{0,1}) =c(i \cdot \mathcal D_{1,0}) + d(i \cdot \mathcal D_{1,1})= a\mathcal D_{0,0}+b \mathcal D_{1,0}+c\mathcal D_{0,1}+d \mathcal D_{1,1}$.
\end{center}
Hence, the $\mathbb Z_4$ invariant subspace of $H^0(\mathcal A_\theta^{alg}, {}_{-1}\mathcal A_\theta^{alg \ast})$ is a three dimensional subspace.
\end{proof}
\begin{thm}  \label{thm:hoch04}
$H^0(\mathcal A_\theta^{alg} \rtimes \mathbb Z_4, (\mathcal A_\theta^{alg} \rtimes \mathbb Z_4)^ \ast) \cong \mathbb C^8$.
\end{thm}
\begin{proof} Using the paracyclic decomposition of the group $H^0(\mathcal A_\theta^{alg} \rtimes \mathbb Z_4, (\mathcal A_\theta^{alg} \rtimes \mathbb Z_4)^ \ast)$  we have the following relation
\begin{center}
$H^0(\mathcal A_\theta^{alg} \rtimes \mathbb Z_4, (\mathcal A_\theta^{alg} \rtimes \mathbb Z_4)^ \ast) = H^0(\mathcal A_\theta, {}_{i}\mathcal A_\theta^{alg \ast})^{\mathbb Z_4} \displaystyle \oplus  H^0(\mathcal A_\theta, {}_{-i}\mathcal A_\theta^{alg \ast})^{\mathbb Z_4} \displaystyle \oplus H^0(\mathcal A_\theta^{alg}, {}_{-1}\mathcal A_\theta^{alg \ast})^{\mathbb Z_4} \displaystyle \oplus  H^0(\mathcal A_\theta, \mathcal A_\theta^{alg \ast})^{\mathbb Z_4} = \mathbb C^2 \displaystyle \oplus \mathbb C^2 \displaystyle \oplus \mathbb C^3 \displaystyle \oplus \mathbb C = \mathbb C^8$
\end{center}
\end{proof}

\centerline{\uline{The case $\Gamma = \mathbb Z_6$.}}
This group is generated in $SL(2,\mathbb Z)$ through its generator $-\omega= \left[
 \begin{array}{cc}
   0 & -1 \\
   1 & 1
 \end{array} \right]\in SL(2,\mathbb Z)$. The generator acts on $\mathcal A_\theta^{alg}$ in the following way
\begin{center}
$U_1 \mapsto U_2, U_2 \mapsto \displaystyle \frac{U_1^{-1}U_2}{\sqrt\lambda}$.
\end{center}
We know that $H^0(\mathcal A_\theta^{alg}, {}_{\omega}\mathcal A_\theta^{alg})$Also we notice that the $-\omega$ twisted Hochschild cohomology group $H^\bullet(\mathcal A_\theta^{alg}, {}_{-\omega}\mathcal A_\theta^{alg})$ and the $-\omega^2$ twisted Hochschild cohomology group $H^\bullet(\mathcal A_\theta^{alg}, {}_{-\omega^2}\mathcal A_\theta^{alg})$ are isomorphic. Hence we consider the case $g=-\omega$ for our computational purpose. \par
For $g=-\omega$ the following is the Hochschild cohomology complex for the Connes resolution.
\begin{center}
${}_{-\omega}\mathcal A_\theta^{alg \ast} \xrightarrow{{}_{-\omega}\alpha_1}{}_{-\omega}\mathcal A_\theta^{alg \ast} \oplus {}_{-\omega}\mathcal A_\theta^{alg \ast} \xrightarrow{{}_{-\omega}\alpha_2} {}_{-\omega}\mathcal A_\theta^{alg \ast} \rightarrow 0$,
\end{center}
where the cochain maps are as follows:

\begin{center}
${}_{-\omega}\alpha_1(\varphi) = (U_2 \varphi - \varphi U_1, \displaystyle\frac{U_1^{-1}U_2 \varphi}{\sqrt{\lambda}} - \varphi
 U_2)$ ; ${}_{-\omega}\alpha_2(\varphi_1, \varphi_2) = \displaystyle\frac{U_1^{-1}U_2 \varphi_1}{\sqrt{\lambda}} - \lambda \varphi_1 U_2 - \lambda U_2 \varphi_2 + \varphi_2 U_1$.
\end{center}
Hence the $-\omega$ twisted Hochschild cohomology group $H^0(\mathcal A_\theta^{alg}, {}_{-\omega}\mathcal A_\theta^{alg \ast})$ is the group $ker({}_{-\omega}\alpha_1)$.
The group $ker({}_{-\omega}\alpha_1)$ is the set of all entries $\varphi \in \mathcal A_\theta^{alg \ast}$ such that the following relations hold:
\begin{center}
 $U_2 \varphi = \varphi U_1$ and $\displaystyle \frac{U_1^{-1}U_2}{\sqrt\lambda} \varphi = \varphi U_2$.
\end{center}
Hence, we deduce that $\varphi_{{n-1},m}\lambda^{m} = \lambda^{n} \varphi_{n,m-1}$ and $\varphi_{n,{m-1}} = \lambda^{n+ \frac{1}{2}} \varphi_{n+1,m-1}$.
\begin{lemma} If $\varphi$ is a cocycle of $H^0(\mathcal A_\theta^{alg}, {}_{-\omega}\mathcal A_\theta^{alg \ast})$ then,
$\varphi$ is generated by the coefficient $\varphi_{0,0}$. 
And for all $(n,m) \in \mathbb Z^2$, $\varphi_{n,m} = \lambda^{-\frac{m^2+n^2}{2}} \varphi_{0,0}$.
 \end{lemma}
\begin{proof}
We observe that all the coefficients $\varphi_{\bullet, r}$ are generated by $\varphi_{0,r}$. Furthermore the relation $\varphi_{{n-1},m}\lambda^{m} = \lambda^{n} \varphi_{n,m-1}$ gives us that $\varphi_{0,\bullet}$ are generated by $\varphi_{0,0}$. Hence the $-\omega$ twisted zeroth Hochschild cohomology group $H^0(\mathcal A_\theta^{alg}, {}_{-\omega}\mathcal A_\theta^{alg \ast})$ is a one dimensional group generated by the coefficient $\varphi_{0,0}$. \par
The second relation gives us that 
$$\varphi_{n,\bullet} = \lambda^{-\frac{n(n-1)}{2} + \frac{(n-1)}{2}} \varphi_{0,\bullet}=\lambda^{\frac{-n^2}{2}} \varphi_{0,\bullet}.$$
Similarly, using both the relations we deduce that $\varphi_{\bullet, m} = \lambda^{-(m-1)+\frac{1}{2}} \varphi_{\bullet, m-1}.$ Hence we have 
$$\varphi_{0,m} = \lambda^{-\frac{m(m-1)}{2} + \frac{(m-1)}{2}} \varphi_{0,0}=\lambda^{\frac{-m^2}{2}} \varphi_{0,0}.$$
Therefore, $\varphi_{n,m} =\lambda^{\frac{-n^2}{2}} \varphi_{0,m}=\lambda^{\frac{-n^2}{2}}\lambda^{\frac{-m^2}{2}} \varphi_{0,0}=\lambda^{-\frac{m^2+n^2}{2}} \varphi_{0,0}$
\end{proof}

\begin{lemma}
$H^0(\mathcal A_\theta^{alg} , {}_{-\omega}\mathcal A_\theta^{alg \ast})^{\mathbb Z_6} \cong \mathbb C$.

\end{lemma}
\begin{proof}
We need to show that the lone cocycle of the group $H^0(\mathcal A_\theta^{alg}, {}_{-\omega}\mathcal A_\theta^{alg \ast})$ is invariant under the action of $-\omega$. The cocycle $\mathcal G^{-\omega}_{0,0}$ generated by $\varphi_{0,0}$ as described in the lemma above has a typical entry of the form $\varphi_{n,m} U_1^n U_2^m$. When $-\omega$ acts on this entry we have the following 
\begin{center}
$-\omega \cdot \varphi_{n,m} U_1^n U_2^m = \varphi_{n,m} U_2^n (\displaystyle \frac{U_1^{-1} U_2}{\sqrt \lambda})^m=\varphi_{n,m} \lambda^{-\frac{m}{2}}U_2^n (U_1^{-1} U_2)^m=\varphi_{n,m} \lambda^{-\frac{m}{2}} \lambda^{-\frac{m(m-1)}{2}}U_2^nU_1^{-m} U_2^m = \varphi_{n,m} \lambda^{-\frac{m^2}{2}} U_2^nU_1^{-m} U_2^m = \varphi_{n,m} \lambda^{-\frac{m^2}{2}} \lambda^{-nm} U_1^{-m} U_2^{n+m}=\varphi_{0,0} \lambda^{-\frac{n^2+m^2}{2}}\lambda^{-\frac{m^2}{2}} \lambda^{-nm} U_1^{-m} U_2^{n+m}=\varphi_{0,0} \lambda^{-\frac{m^2+(n+m)^2}{2}}  U_1^{-m} U_2^{n+m}=\varphi_{{-m},n+m} U_1^{-m} U_2^{n+m}$.
\end{center}
Hence we see that the cocycle $\mathcal G^{-\omega}_{0,0}$ is invariant under the action of $\mathbb Z_6$.
\end{proof}
\begin{lemma}
$H^0(\mathcal A_\theta^{alg} , {}_{\omega}\mathcal A_\theta^{alg \ast})^{\mathbb Z_6} \cong \mathbb C^2$.
\end{lemma}

\begin{proof}
The actions of $-\omega$ on the twisted zero Hochschild cocycles $\mathcal E^\omega_{0,0}$, $\mathcal E^\omega_{0,1}$ and $\mathcal E^\omega_{0,-1}$ are described below. \par
 We consider $\mathcal E^\omega_{0,0}$, a typical entry $\varphi_{n,m} U_1^n U_2^m$ when acted by $-\omega$ transforms to
\begin{center} $\varphi_{n,m} U_2^n (\displaystyle \frac{U_1^{-1}U_2}{\sqrt \lambda})^m=\varphi_{n,m} \lambda^{-\frac{m}{2}}U_2^n (U_1^{-1} U_2)^m=\varphi_{n,m} \lambda^{-\frac{m}{2}} \lambda^{-\frac{m(m-1)}{2}}U_2^nU_1^{-m} U_2^m = \varphi_{n,m} \lambda^{-\frac{m^2}{2}} U_2^nU_1^{-m} U_2^m =  \varphi_{n,m} \lambda^{-\frac{m^2}{2}} \lambda^{-nm} U_1^{-m} U_2^{n+m}=\varphi_{0,0} \lambda^{\frac{n^2+m^2+4nm}{6}} \lambda^{-\frac{m^2}{2}} \lambda^{-nm} U_1^{-m} U_2^{n+m}=\varphi_{0,0} \lambda^{\frac{n^2-2m^2-2nm}{6}}  \lambda^{-nm} U_1^{-m} U_2^{n+m}=\varphi_{-m,n+m} U_1^{-m} U_2^{n+m}$.
\end{center}
Since for the cocycle $\mathcal E^\omega_{0,0}$, $(n+m)-(-m) - (m-n) = m+2n = m-n+3n \equiv m-n \pmod{3} \equiv 0 \pmod{3}$. Hence, $-\omega \cdot \varphi_{n,m} U_1^n U_2^m $ is an entry in $\mathcal E^\omega_{0,0}$, so, $-\omega \cdot \mathcal E^\omega_{0,0} = \mathcal E^\omega_{0,0}$. \par
Similarly for the cocycle $\mathcal E^\omega_{0,1}$we have
\begin{center}
$\varphi_{n,m} U_2^n (\displaystyle \frac{U_1^{-1}U_2}{\sqrt \lambda})^m=\varphi_{n,m} \lambda^{-\frac{m}{2}}U_2^n (U_1^{-1} U_2)^m=\varphi_{n,m} \lambda^{-\frac{m}{2}} \lambda^{-\frac{m(m-1)}{2}}U_2^nU_1^{-m} U_2^m = \varphi_{n,m} \lambda^{-\frac{m^2}{2}} U_2^nU_1^{-m} U_2^m =  \varphi_{n,m} \lambda^{-\frac{m^2}{2}} \lambda^{-nm} U_1^{-m} U_2^{n+m}=\varphi_{0,1} \lambda^{\frac{n^2+m^2+4nm-1}{6}} \lambda^{-\frac{m^2}{2}} \lambda^{-nm} U_1^{-m} U_2^{n+m}=\varphi_{0,1} \lambda^{\frac{n^2-2m^2-2nm-1}{6}}  \lambda^{-nm} U_1^{-m} U_2^{n+m}=\varphi_{-m,n+m} U_1^{-m} U_2^{n+m}$.
\end{center}
Unlike the previous case, for the cocycle $\mathcal E^\omega_{0,1}$ we have 
\begin{center}
$(n+m)-(-m) - (m-n) = m+2n = m-n+3n \equiv m-n \pmod{3} \equiv 1 \pmod{3}$.
\end{center}
\begin{center}
Hence, $-\omega \cdot \mathcal E^\omega_{0,1} = \mathcal E^\omega_{0,-1}$. 
\end{center}
Similarly, it is trivial to figure that  $-\omega \cdot \mathcal E^\omega_{0,-1} = \mathcal E^\omega_{0,1}$. Hence, the two dimensional $-\omega$ invariant subspace of $H^0(\mathcal A_\theta^{alg},{}_{\omega} \mathcal A_\theta^{alg \ast})$ is generated by $\mathcal E^\omega_{0,0}$ and $\mathcal E^\omega_{0,1} + \mathcal E^\omega_{0,-1}$.

\end{proof}

\begin{lemma}
$H^0(\mathcal A_\theta^{alg} , {}_{-1}\mathcal A_\theta^{alg \ast})^{\mathbb Z_6} \cong \mathbb C^2$.
\end{lemma}

\begin{proof}
The four $-1$ twisted zero Hochschild cohomology cocycles are $\mathcal D_{0,0}$, $\mathcal D_{0,1}$, $\mathcal D_{1,0}$ and $\mathcal D_{1,1}$. Consider the action of $-\omega$ on $\mathcal D_{0,0}$,
\begin{center}
 $-1 \cdot \varphi_{2n,2m} U_1^{2n} U_2^{2m} = \varphi_{2n,2m} U_2^{2n} (\displaystyle \frac{U_1^{-1}U_2}{\sqrt \lambda})^{2m}=\varphi_{2n,2m} U_2^{2n} \lambda^{-m}(U_1^{-1}U_2)^{2m}=\varphi_{2n,2m} \lambda^{-m} \lambda^{-\frac{2m(2m-1)}{2}}U_2^{2n}U_1^{-2m} U_2^{2m}=\varphi_{2n,2m} \lambda^{-2m^2} \lambda^{-4nm} U_1^{-2m} U_2^{2n+2m}=\varphi_{0,0} \lambda^{2nm} \lambda^{-2m^2} \lambda^{-4nm} U_1^{-2m} U_2^{2n+2m}=\varphi_{0,0} \lambda^{-2m^2-2nm} U_1^{-2m} U_2^{2n+2m}= \varphi_{{-2m},{2n+2m}} U_1^{-2m} U_2^{2n+2m}$. 
\end{center}
Hence, we conclude that $\mathcal D_{0,0}$ is invariant under the action of $-\omega$.\par
Similar computations for $\mathcal D_{0,1}$, $\mathcal D_{1,0}$ and $\mathcal D_{1,1}$ yield the following
\begin{center}
$-\omega \cdot \mathcal D_{0,1} = \sqrt \lambda \mathcal D_{1,1}$ ;
$-\omega \cdot \mathcal D_{1,0} = \mathcal D_{0,1}$ ;
$-\omega \cdot \mathcal D_{1,1} = \lambda \mathcal D_{1,0}$.
\end{center}
Hence the $-\omega$ invariant sub-space of $H^0(\mathcal A_\theta^{alg}, {}_{-1}\mathcal A_\theta^{alg \ast})$ is $2$ dimensional and is generated by $\mathcal D_{0,0}$ and $\mathcal D_{1,0}+\lambda \sqrt{\lambda}\mathcal D_{1,0}+\sqrt\lambda \mathcal D_{1,1}$.
\end{proof}

\begin{thm}  \label{thm:hoch06}
$H^0(\mathcal A_\theta^{alg} \rtimes \mathbb Z_6, (\mathcal A_\theta^{alg} \rtimes \mathbb Z_6)^ \ast) \cong \mathbb C^9$.
\end{thm}
\begin{proof} Using the paracyclic decomposition of the group $H^0(\mathcal A_\theta^{alg} \rtimes \mathbb Z_6, (\mathcal A_\theta^{alg} \rtimes \mathbb Z_6)^ \ast)$  we have the following relation
\begin{center}
$H^0(\mathcal A_\theta^{alg} \rtimes \mathbb Z_6, (\mathcal A_\theta^{alg} \rtimes \mathbb Z_6)^ \ast) = H^0(\mathcal A_\theta, {}_{-\omega}\mathcal A_\theta^{alg \ast})^{\mathbb Z_6} \displaystyle \oplus  H^0(\mathcal A_\theta, {}_{-\omega^2}\mathcal A_\theta^{alg \ast})^{\mathbb Z_6} \displaystyle \oplus H^0(\mathcal A_\theta^{alg}, {}_{\omega}\mathcal A_\theta^{alg \ast})^{\mathbb Z_6} \displaystyle \oplus  H^0(\mathcal A_\theta, {}_{\omega^2}\mathcal A_\theta^{alg \ast})^{\mathbb Z_6}\displaystyle \oplus  H^0(\mathcal A_\theta, {}_{-1}\mathcal A_\theta^{alg \ast})^{\mathbb Z_6}\displaystyle \oplus  H^0(\mathcal A_\theta, \mathcal A_\theta^{alg \ast})^{\mathbb Z_6} = \mathbb C \displaystyle \oplus \mathbb C \displaystyle \oplus \mathbb C^2 \displaystyle \oplus \mathbb C^2 \displaystyle \oplus \mathbb C^2 \displaystyle \oplus \mathbb C = \mathbb C^9$
\end{center}
\end{proof}
\subsection{The Hochschild cohomology groups $H^2(\mathcal A_\theta^{alg} \rtimes \Gamma, (\mathcal A_\theta^{alg} \rtimes \Gamma)^{\ast}))$}~\\

In this section we study the second Hochschild cohomology groups of the $\mathbb Z_3$, $\mathbb Z_4$ and $\mathbb Z_6$ noncommutative algebraic torus orbifold. From \cite{Q2} we have $H^2(\mathcal A_\theta^{alg}, {}_{-1}\mathcal A_\theta^{alg \ast})^{\mathbb Z_2} = 0$. We conclude that 
$$H^2(\mathcal A_\theta^{alg}, {}_{-1}\mathcal A_\theta^{alg \ast})^{\mathbb Z_4} =H^2(\mathcal A_\theta^{alg}, {}_{-1}\mathcal A_\theta^{alg \ast})^{\mathbb Z_6}= 0.$$
Similarly we also conclude that
$$H^2(\mathcal A_\theta^{alg}, {}_{\omega^2}\mathcal A_\theta^{alg \ast})^{\mathbb Z_6} =H^2(\mathcal A_\theta^{alg}, {}_{\omega}\mathcal A_\theta^{alg \ast})^{\mathbb Z_6}= 0.$$
\centerline{\uline{The case $\Gamma = \mathbb Z_3$.}}
We notice that the $\omega$ twisted Hochschild cohomology group $H^2(\mathcal A_\theta^{alg}, {}_{\omega}\mathcal A_\theta^{alg \ast})={}_{\omega}\mathcal A_\theta^{alg \ast} / Im({}_{\omega}\alpha_2)$ is generated by the classes $\varphi_{0,0}$(generated by $U_1^{0}U_2^{0}), \varphi_{1,0}$(generated by $U_1^{-1}U_2^{0})$ and $\varphi_{0,1}$(generated by $U_1^{0}U_2^{-1})$. \par
We argue the case as in \cite{Q2}, for $\varphi \in {}_{\omega}\mathcal A_\theta^{alg \ast}$ and let $\widetilde{\varphi}$ be the corresponding element of $\text{Hom}_{\mathcal B_\theta^{alg}}(J_2,{}_{\omega}\mathcal A_\theta^{alg\ast})$. Where $J_\ast$ is the chain complex $J_{\ast, \omega}^{\mathbb Z_3}$ defined in \cite[Page 340, Section 8]{Q1} and  $\mathcal B_\theta^{alg} = \mathcal A_\theta^{alg} \otimes (\mathcal A_\theta^{alg})^{op}$. We have the following relation
$$\widetilde{\varphi}(a\otimes b \otimes e_1 \wedge e_2)(x)=\varphi((\omega\cdot b)xa),$$
for all  $a,b,x \in \mathcal A_\theta^{alg}$.
Let $\psi = k_2^{\ast} \widetilde{\varphi} = \widetilde{\varphi} \circ k_2$. We also have
$$\psi(x,x_1,x_2)=\widetilde{\varphi}(k_2(I \otimes x_1 \otimes x_2))(x),$$
for all $x,x_1,x_2 \in \mathcal A_\theta^{alg}$.
The group $\omega$ acts on $\mathcal A_\theta^{alg}$ in the bar complex as $$\omega \cdot \chi(x,x_1,x_2)=\chi(\omega\cdot x,\omega \cdot x_1, \omega \cdot x_2).$$ 
Further we pull the map ${}_{\omega}\psi := \omega \cdot \psi$ back on to the Connes complex via the map $h_2^{\ast}$. Let $w=h_2^{\ast}({}_{\omega}\psi)$ denote the pullback of ${}_{\omega}\psi$ on the Connes complex. We have  
\begin{center}
$w(x)={}_{\omega}\psi(x,U_2,U_1)-\lambda{}_{\omega}\psi(x,U_1,U_2)=\psi(\omega \cdot x, \frac{U_1U_2^{-1}}{\sqrt\lambda},U_2^{-1}) -\lambda \psi(\omega \cdot x, U_2^{-1},\frac{U_1U_2^{-1}}{\sqrt\lambda})= \newline
\widetilde{\varphi}(k_2(I \otimes \frac{U_1U_2^{-1}}{\sqrt\lambda}\otimes U_2^{-1}))(\omega\cdot x)-\lambda \widetilde{\varphi}(k_2(I \otimes U_2^{-1}\otimes \frac{U_1U_2^{-1}}{\sqrt\lambda}))(\omega\cdot x)$.
\end{center}
Following the calculations from \cite{C} and \cite[Section 6]{Q1}, we have
$$k_2(I \otimes U_1U_2^{-1}\otimes U_2^{-1})-\lambda k_2(I \otimes U_2^{-1}\otimes U_1U_2^{-1})= (U_2^{-1} \otimes U_2^{-2}).$$
Applying this we conclude that
\begin{center}
$ \displaystyle\frac{1}{\sqrt\lambda}(\widetilde{\varphi}((k_2(I \otimes U_1U_2^{-1}\otimes U_2^{-1}))(\omega\cdot x)-\lambda \widetilde{\varphi}(k_2(I \otimes U_2^{-1}\otimes U_1U_2^{-1})))(\omega\cdot x))=\displaystyle\frac{1}{\sqrt\lambda}\widetilde{\varphi}((U_2^{-1} \otimes U_2^{-2}))(\omega\cdot x) = {\sqrt\lambda}\varphi(U_1^{-2 }U_2^{2} \cdot (\omega \cdot x)\cdot U_2^{-1})$.
\end{center}
To assert the $\mathbb Z_3$ invariance of a 2-cocycle $\varphi$, we need to show that for $x \in \mathcal A_\theta^{alg}$, $\varphi(x)$ and ${\sqrt\lambda}\varphi(U_1^{-2 }U_2^{2} \cdot (\omega \cdot x)\cdot U_2^{-1})$ are the same. \par

We consider the 2-cocycle $\varphi_{0,0}$. On one hand we have $\varphi_{0,0}(x) = x_{0,0}$ while on the other hand
\begin{center}
${\sqrt\lambda}\varphi_{0,0}(U_1^{-2 }U_2^{2} \cdot (\omega \cdot (x_{-1,2}U_1^{-1}U_2^2))\cdot U_2^{-1}) =  \displaystyle\frac{1}{\sqrt\lambda}\varphi(x_{-1,2}U_1^{-2 }U_2^{2} \cdot (U_2 U_1^2 U_2^{-2})\cdot U_2^{-1})=\displaystyle\frac{1}{\sqrt\lambda}\varphi(x_{-1,2} U_1^{-2 }U_2^{3} U_1^2 U_2^{-3} ) = \displaystyle\frac{1}{\sqrt\lambda}\lambda^6\varphi(x_{-1,2} U_1^{0 }U_2^{0}) = \lambda^5 \sqrt\lambda x_{-1,2}$.
\end{center}
The cocycle $\varphi_{0,0}$  and $\lambda^3 \sqrt\lambda \varphi_{-1,2}$ represent the same cocycle as they are separated by a coboundary element. Hence we conclude that the 2-cocycle $\varphi_{0,0}$ is not invariant under the action of $\omega$. Therefore $\varphi_{0,0} \notin H^2(\mathcal A_\theta^{alg}, {}_{\omega}\mathcal A_\theta^{alg \ast})^{\mathbb Z_3}$. 
Repeating the above computations for the cocycles $\varphi_{1,0}$ and $\varphi_{0,1}$ we finally conclude that $H^2(\mathcal A_\theta^{alg}, {}_{\omega}\mathcal A_\theta^{alg \ast})^{\mathbb Z_3} = 0$. \par
Now we consider the group $H^2(\mathcal A_\theta^{alg}, {}_{}\mathcal A_\theta^{alg \ast})$ generated by $U_2U_1$\cite[Lemma 3.2]{Q2}. To check the invariance of the untwisted 2-cocycle (say $\varphi_{-1,-1}$), we need to compare $\varphi_{-1,-1}(x)$ and $\displaystyle\frac{1}{\sqrt\lambda}\widetilde{\varphi}((U_2^{-2} (\omega\cdot x)  U_2^{-1}))$ for all $x \in \mathcal A_\theta^{alg}$. If they are the same, the cocycle is $\mathbb Z_3$ invariant. else not. \par

On one hand we have $\varphi_{-1,-1}(x) = x_{-1,-1}$ and on the other hand
\begin{center}
$\displaystyle\frac{1}{\sqrt\lambda}\varphi_{-1,-1}(U_2^{-2} (\omega\cdot x)  U_2^{-1})= \displaystyle\frac{1}{\sqrt\lambda}\varphi_{-1,-1}(U_2^{-2} (\omega\cdot (x_{-1,-1} U_1^{-1} U_2^{-1}))  U_2^{-1})=\displaystyle\frac{1}{\sqrt\lambda}\varphi_{-1,-1}(U_2^{-2} x_{-1,-1} U_2 {(\frac{U_1 U_2^{-1}}{\sqrt\lambda})}^{-1}  U_2^{-1})=\varphi_{-1,-1}(x_{-1,-1} U_2^{-2} U_2 U_2 U_1^{-1} U_2^{-1}) = \varphi_{-1,-1}(x_{-1,-1} U_1^{-1} U_2^{-1}) = x_{-1,-1}$.
\end{center}

Hence the cocycle is invariant under the action of the group $\mathbb Z_3$ and $H^2(\mathcal A_\theta^{alg}, {}_{}\mathcal A_\theta^{alg \ast})^{\mathbb Z_3} \cong \mathbb C$.

\begin{thm}  \label{thm:hoch23}
$H^2(\mathcal A_\theta^{alg} \rtimes \mathbb Z_3, (\mathcal A_\theta^{alg} \rtimes \mathbb Z_3)^ \ast) \cong \mathbb C$.
\end{thm}
\begin{proof} Using the paracyclic decomposition of the group $H^2(\mathcal A_\theta^{alg} \rtimes \mathbb Z_3, (\mathcal A_\theta^{alg} \rtimes \mathbb Z_3)^ \ast)$  we have the following relation
\begin{center}
$H^2(\mathcal A_\theta^{alg} \rtimes \mathbb Z_3, (\mathcal A_\theta^{alg} \rtimes \mathbb Z_3)^ \ast) = H^2(\mathcal A_\theta, {}_{\omega}\mathcal A_\theta^{alg \ast})^{\mathbb Z_3} \displaystyle \oplus  H^2(\mathcal A_\theta, {}_{\omega^2}\mathcal A_\theta^{alg \ast})^{\mathbb Z_3} \displaystyle \oplus H^2(\mathcal A_\theta^{alg}, {}_{}\mathcal A_\theta^{alg \ast})^{\mathbb Z_3}= 0 \displaystyle \oplus 0 \displaystyle \oplus \mathbb C = \mathbb C$
\end{center}
\end{proof}

\centerline{\uline{The case $\Gamma = \mathbb Z_4$.}}
We notice from \cite{Q2} that $H^2(\mathcal A_\theta, {}_{-1}\mathcal A_\theta^{alg \ast})^{\mathbb Z_2}= H^2(\mathcal A_\theta, {}_{-1}\mathcal A_\theta^{alg \ast})^{\mathbb Z_4} = 0$ and the $i$ twisted Hochschild cohomology group $H^2(\mathcal A_\theta^{alg}, {}_{i}\mathcal A_\theta^{alg \ast})={}_{i}\mathcal A_\theta^{alg \ast} / Im({}_{i}\alpha_2)$ is generated by the classes $\varphi_{0,0}$(generated by $U_1^{0}U_2^{0})$ and $\varphi_{1,0}$(generated by $U_1^{-1}U_2^{0})$. \par
As in the previous case and using the results of \cite{C} and  \cite[Section 6]{Q1} we are finally down to comparing $\varphi(x)$ with the following
$$\widetilde\varphi(U_2^{-1} \otimes  U_2^{-1})(i \cdot x).$$
On one hand we have $\varphi_{0,0}(x) = x_{0,0}$ while on the other hand
\begin{center}
$\varphi(U_1^{-1} (i \cdot x) U_2^{-1}) =\varphi(U_1^{-1} (i \cdot (x_{-1,1}U_1^{-1}U_2)) U_2^{-1})=\varphi(x_{-1,1}U_1^{-1} (i \cdot (U_1^{-1}U_2)) U_2^{-1})=\varphi(x_{-1,1}U_1^{-1} U_2U_1 U_2^{-1})= \lambda x_{-1,1}$.
\end{center}
But since  $\varphi_{0,0}$ and $ \varphi_{-1,1}$ represent the same cocycle being separated by a coboundary element. We conclude that the 2-cocycle $\varphi_{0,0}$ is not invariant under the action of $i$. Hence $\varphi_{0,0} \notin H^2(\mathcal A_\theta^{alg}, {}_{i}\mathcal A_\theta^{alg \ast})^{\mathbb Z_4}$. 
Repeating the above computations for the cocycle $\varphi_{1,0}$ we observe that $H^2(\mathcal A_\theta^{alg}, {}_{i}\mathcal A_\theta^{alg \ast})^{\mathbb Z_4} = 0$. \par
We consider the group $H^2(\mathcal A_\theta^{alg}, {}_{}\mathcal A_\theta^{alg \ast})$ generated by $U_2U_1$. To check the invariance of this untwisted 2-cocycle $\varphi_{-1,-1}$, we need to compare $\varphi_{-1,-1}(x)$ and $\varphi_{-1,-1}(U_2^{-1} (i \cdot x) U_2^{-1})$ for all $x \in \mathcal A_\theta^{alg}$. If they are the same, the cocycle is $\mathbb Z_4$ invariant, else not. \par

On one hand we have $\varphi_{-1,-1}(x) = x_{-1,-1}$ and on the other hand
\begin{center}
$\varphi_{-1,-1}(U_2^{-1} (i \cdot x) U_2^{-1})=\varphi_{-1,-1}(U_2^{-1} (i \cdot (x_{-1,-1} U_1^{-1} U_2^{-1})) U_2^{-1})=\varphi_{-1,-1}(x_{-1,-1}U_2^{-1} (i \cdot (U_1^{-1} U_2^{-1})) U_2^{-1})=\varphi_{-1,-1}(x_{-1,-1}U_2^{-1} U_2  U_1^{-1} U_2^{-1})= x_{-1,-1}.$
\end{center}

Hence the cocycle is invariant under the action of the group $\mathbb Z_3$ and $H^2(\mathcal A_\theta^{alg}, {}_{}\mathcal A_\theta^{alg \ast})^{\mathbb Z_4} \cong \mathbb C$.

\begin{thm}  \label{thm:hoch24}
$H^2(\mathcal A_\theta^{alg} \rtimes \mathbb Z_4, (\mathcal A_\theta^{alg} \rtimes \mathbb Z_4)^ \ast) \cong \mathbb C$.
\end{thm}
\begin{proof} Using the paracyclic decomposition of the group $H^2(\mathcal A_\theta^{alg} \rtimes \mathbb Z_4, (\mathcal A_\theta^{alg} \rtimes \mathbb Z_4)^ \ast)$  we have the following relation
\begin{center}
$H^2(\mathcal A_\theta^{alg} \rtimes \mathbb Z_4, (\mathcal A_\theta^{alg} \rtimes \mathbb Z_4)^ \ast) = H^2(\mathcal A_\theta, {}_{i}\mathcal A_\theta^{alg \ast})^{\mathbb Z_4} \displaystyle \oplus  H^2(\mathcal A_\theta, {}_{-i}\mathcal A_\theta^{alg \ast})^{\mathbb Z_4} \displaystyle \oplus  H^2(\mathcal A_\theta, {}_{-1}\mathcal A_\theta^{alg \ast})^{\mathbb Z_4} \displaystyle \oplus H^2(\mathcal A_\theta^{alg}, {}_{}\mathcal A_\theta^{alg \ast})^{\mathbb Z_4}= 0 \displaystyle \oplus 0 \displaystyle \oplus 0 \displaystyle \oplus \mathbb C = \mathbb C$
\end{center}
This completes the proof.
\end{proof}

\centerline{\uline{The case $\Gamma = \mathbb Z_6$.}}
We notice from preceding theorems that 
\begin{center}
$H^2(\mathcal A_\theta, {}_{\omega}\mathcal A_\theta^{alg \ast})^{\mathbb Z_3}= H^2(\mathcal A_\theta, {}_{\omega}\mathcal A_\theta^{alg \ast})^{\mathbb Z_6} =H^2(\mathcal A_\theta, {}_{\omega^2}\mathcal A_\theta^{alg \ast})^{\mathbb Z_3}= H^2(\mathcal A_\theta, {}_{\omega^2}\mathcal A_\theta^{alg \ast})^{\mathbb Z_6}=  H^2(\mathcal A_\theta, {}_{-1}\mathcal A_\theta^{alg \ast})^{\mathbb Z_3} =H^2(\mathcal A_\theta, {}_{-1}\mathcal A_\theta^{alg \ast})^{\mathbb Z_6}= H^2(\mathcal A_\theta, {}_{\omega}\mathcal A_\theta^{alg \ast})^{\mathbb Z_6}=0$.
\end{center}
The $-\omega$ twisted Hochschild cohomology group $H^2(\mathcal A_\theta^{alg}, {}_{-\omega}\mathcal A_\theta^{alg \ast})={}_{-\omega}\mathcal A_\theta^{alg \ast} / Im({}_{-\omega}\alpha_2)$ is generated by the classes $\varphi_{0,0}$. \par
As in both the previous cases, to check the $\mathbb Z_6$ invariance of an $-\omega$ twisted Hochschild cocycle $\varphi$, we compare $\varphi(x)$ with the following
$$\frac{1}{\sqrt\lambda}\widetilde\varphi(U_1^{-1} \otimes U_1^{-1}U_2)(-\omega \cdot x).$$
If $\varphi(x)$ equals $\frac{1}{\sqrt\lambda}\widetilde\varphi(U_1^{-1} \otimes U_1^{-1}U_2)(-\omega \cdot x)$, we say that the the cocycle $\varphi$ is $\mathbb Z_6$ invariant, else not.
Now let us consider the 2-cocycle $\varphi_{0,0}$. On one hand we have $\varphi_{0,0}(x) = x_{0,0}$ while on the other hand
\begin{center}
$\frac{1}{\sqrt\lambda}\widetilde{\varphi_{0,0}}(U_1^{-1} \otimes U_1^{-1}U_2)(-\omega \cdot x) = \frac{1}{\lambda} \varphi_{0,0}(U_2^{-1} U_1^{-1} U_2 (-\omega \cdot x) U_1^{-1}) = \frac{1}{\lambda^2} \varphi_{0,0}(U_1^{-1} (-\omega \cdot x) U_1^{-1}))=\frac{1}{\lambda} x_{2,-2}$.
\end{center}
Since in this case the cocycle $\varphi_{0,0}$ and $\lambda^{2}\varphi_{2,-2}$ are the same cocycle. We conclude that the 2-cocycle $\varphi_{0,0}$ is not invariant under the action of $-\omega$. Hence we observe that $$H^2(\mathcal A_\theta^{alg}, {}_{-\omega}\mathcal A_\theta^{alg \ast})^{\mathbb Z_6} = 0.$$
Now we check the invariance the group $H^2(\mathcal A_\theta^{alg}, {}_{}\mathcal A_\theta^{alg \ast})$ generated by $U_2U_1$. For this we need to compare $\varphi_{-1,-1}(x)$ and ${\sqrt\lambda}\widetilde{\varphi}((U_1^{-1} (-\omega\cdot x)  U_1^{-1}U_2))$ for all $x \in \mathcal A_\theta^{alg}$. If they are the same, the cocycle is $\mathbb Z_6$ invariant. else not. \par

On one hand we have $\varphi_{-1,-1}(x) = x_{-1,-1}$ and on the other hand
\begin{center}
$ \frac{1}{\sqrt\lambda}{\varphi_{-1,-1}}((U_1^{-1} U_2 (-\omega\cdot x)  U_1^{-1}))=\frac{1}{\sqrt\lambda}{\varphi_{-1,-1}}((U_1^{-1} (-\omega\cdot (x_{-1,-1} U_1^{-1} U_2^{-1}))  U_1^{-1}U_2))=\frac{1}{\sqrt\lambda}{\varphi_{-1,-1}}(x_{-1,-1} (U_1^{-1} (-\omega\cdot (U_1^{-1} U_2^{-1}))  U_1^{-1}U_2))=\frac{1}{\sqrt\lambda}\varphi_{-1,-1}(x_{-1,-1}U_1^{-1} U_2 U_2^{-1} (\frac{U_1^{-1}U_2}{\sqrt\lambda})^{-1}U_1^{-1})=\varphi_{-1,-1}(x_{-1,-1}U_1^{-1} U_2^{-1} )= x_{-1,-1}$.
\end{center}

Hence the cocycle is invariant under the action of the group $\mathbb Z_6$ and $H^2(\mathcal A_\theta^{alg}, {}_{}\mathcal A_\theta^{alg \ast})^{\mathbb Z_6} \cong \mathbb C$.

\begin{thm}  \label{thm:hoch26}
$H^2(\mathcal A_\theta^{alg} \rtimes \mathbb Z_6, (\mathcal A_\theta^{alg} \rtimes \mathbb Z_6)^ \ast) \cong \mathbb C$.
\end{thm}
\begin{proof} Using the paracyclic decomposition of the group $H^2(\mathcal A_\theta^{alg} \rtimes \mathbb Z_6, (\mathcal A_\theta^{alg} \rtimes \mathbb Z_6)^ \ast)$  we have the following relation
\begin{center}
$H^2(\mathcal A_\theta^{alg} \rtimes \mathbb Z_6, (\mathcal A_\theta^{alg} \rtimes \mathbb Z_6)^ \ast) = H^2(\mathcal A_\theta, {}_{-\omega}\mathcal A_\theta^{alg \ast})^{\mathbb Z_6} \displaystyle \oplus  H^2(\mathcal A_\theta, {}_{-\omega^2}\mathcal A_\theta^{alg \ast})^{\mathbb Z_6} \displaystyle \oplus  H^2(\mathcal A_\theta, {}_{\omega^2}\mathcal A_\theta^{alg \ast})^{\mathbb Z_6}\displaystyle \oplus  H^2(\mathcal A_\theta, {}_{\omega}\mathcal A_\theta^{alg \ast})^{\mathbb Z_6}\displaystyle \oplus  H^2(\mathcal A_\theta, {}_{-1}\mathcal A_\theta^{alg \ast})^{\mathbb Z_6} \displaystyle \oplus H^2(\mathcal A_\theta^{alg}, {}_{}\mathcal A_\theta^{alg \ast})^{\mathbb Z_6}= 0 \displaystyle \oplus 0 \displaystyle \oplus 0 \displaystyle \oplus 0 \displaystyle \oplus 0 \displaystyle \oplus \mathbb C = \mathbb C$
\end{center}
\end{proof}

\subsection{The Hochschild cohomology groups $H^1(\mathcal A_\theta^{alg} \rtimes \Gamma, (\mathcal A_\theta^{alg} \rtimes \Gamma)^{\ast}))$}~\\
\centerline{\uline{The case $\Gamma = \mathbb Z_3$.}}
In this case we have the following paracyclic decomposition of the first Hochschild cohomology group for $\mathcal A_\theta^{alg} \rtimes \mathbb Z_3$.
$$H^1(\mathcal A_\theta^{alg} \rtimes \mathbb Z_3, (\mathcal A_\theta^{alg} \rtimes \mathbb Z_3)^\ast)=H^1(\mathcal A_\theta^{alg}, \mathcal A_\theta^{alg \ast})^{\mathbb Z_3} \displaystyle \oplus H^1(\mathcal A_\theta^{alg}, {}_{\omega}\mathcal A_\theta^{alg \ast})^{\mathbb Z_3}\displaystyle \oplus H^1(\mathcal A_\theta^{alg}, {}_{\omega^2}\mathcal A_\theta^{alg \ast})^{\mathbb Z_3}.$$
We observe that $H^1(\mathcal A_\theta^{alg}, \mathcal A_\theta^{alg \ast}) = \mathbb C^2$, generated by $\varphi^1_{-1,0}$ and $\varphi^2_{0,-1}$ \cite{C}. We need to identify its $\mathbb Z_3$ invariant subgroup. For $a,b \in \mathbb C$, we consider the cocycle $\chi := (a\varphi^1_{-1,0},b\varphi^2_{0,-1}) \in \mathcal A_\theta^{alg \ast} \oplus \mathcal A_\theta^{alg \ast}$ and push forward to the bar complex, thereafter we pull it back to the bar complex after the $\mathbb Z_3$ action of  it. Proceeding in a similar way as in \cite{Q2} we finally observe that $h_1^{\ast}(\omega \cdot(  k_{1 \ast} \chi))(x) = (-a\lambda\sqrt\lambda x_{-1,-1}, \frac{b}{\sqrt\lambda} x_{0,0}-b\lambda^{-2} x_{2,-1}) $. Hence we conclude that $H^1(\mathcal A_\theta^{alg}, \mathcal A_\theta^{alg \ast})^{\mathbb Z_3} = 0$. Since we have $H^1(\mathcal A_\theta^{alg}, \mathcal A_\theta^{alg \ast})^{\mathbb Z_2} = H^1(\mathcal A_\theta^{alg}, \mathcal A_\theta^{alg \ast})^{\mathbb Z_3} = 0$, we conclude that 
$$H^1(\mathcal A_\theta^{alg}, \mathcal A_\theta^{alg \ast})^{\mathbb Z_4} = H^1(\mathcal A_\theta^{alg}, \mathcal A_\theta^{alg \ast})^{\mathbb Z_6} = 0.$$
Now we turn our attention to understand the group $H^1(\mathcal A_\theta^{alg}, {}_{\omega}\mathcal A_\theta^{alg \ast})$. We deploy the same method which we had used to calculate $H^1(\mathcal A_\theta^{alg}, {}_{-1}\mathcal A_\theta^{alg \ast})$ in \cite{Q2}. Below, a typical portion of the infinite graph corresponding to a ${}_{\omega}\alpha_2$ kernel equation is plotted on the $\mathbb Z^2$ lattice plane.
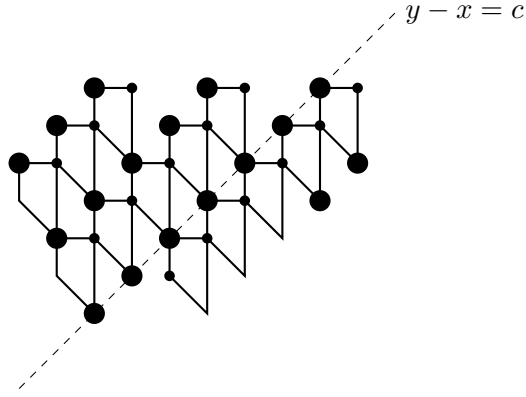
\begin{center}
\begin{tikzpicture}
\fill (-.5,0.5) circle (2pt);
\fill (1,0.5) circle (2pt);
\fill (2.5,0.5) circle (2pt);
\fill (-1,.5) circle (4pt);
\fill (0.5,.5) circle (4pt);
\fill (2,.5) circle (4pt);
\fill (-1,0) circle (2pt);
\fill (0.5,0) circle (2pt);
\fill (2,0) circle (2pt);
\fill (-1.5,0) circle (4pt);
\fill (0,0) circle (4pt);
\fill (1.5,0) circle (4pt);
\fill (-1.5,-0.5) circle (2pt);
\fill (0,-0.5) circle (2pt);
\fill (1.5,-0.5) circle (2pt);
\fill (-2,-0.5) circle (4pt);
\fill (-0.5,-0.5) circle (4pt);
\fill (1,-0.5) circle (4pt);
\fill (-1,-1) circle (4pt);
\fill (-1.5,-1.5) circle (4pt);
\fill (0.5,-1) circle (4pt);
\fill (0,-1.5) circle (4pt);
\fill (-0.5,-2) circle (4pt);
\fill (2.5,-.5) circle (4pt);
\fill (-1,-2.5) circle (4pt);
\fill (2,-1) circle (4pt);
\fill (1,-1) circle (2pt);
\fill (.5,-1.5) circle (2pt);
\fill (0,-2) circle (2pt);
\fill (-.5,-1) circle (2pt);
\fill (-1,-1.5) circle (2pt);
\draw[thick](-1,0.5)--(-.5,0.5)--(-.5,-0.5)--(-1,0)--(-1,.5)node[right]{};
\draw[thick](.5,0.5)--(1,0.5)--(1,-0.5)--(.5,0)--(.5,.5)node[right]{};
\draw[thick](2,0.5)--(2.5,0.5)--(2.5,-0.5)--(2,0)--(2,.5)node[right]{};
\draw[thick](1.5,0)--(2,0)--(2,-1)--(1.5,-0.5)--(1.5,0)node[right]{};
\draw[thick](0,0)--(.5,0)--(.5,-1)--(0,-0.5)--(0,0)node[right]{};
\draw[thick](-0.5,-0.5)--(0,-0.5)--(0,-1.5)--(-0.5,-1)--(-.5,-.5)node[right]{};
\draw[thick](-1.5,0)--(-1,0)--(-1,-1)--(-1.5,-0.5)--(-1.5,0)node[right]{};
\draw[thick](-2,-0.5)--(-1.5,-0.5)--(-1.5,-1.5)--(-2,-1)--(-2,-0.5)node[right]{};
\draw[thick](-1,-1)--(-0.5,-1)--(-0.5,-2)--(-1,-1.5)--(-1,-1)node[right]{};
\draw[thick](-1.5,-1.5)--(-1,-1.5)--(-1,-2.5)--(-1.5,-2)--(-1.5,-1.5)node[right]{};
\draw[thick](1,-.5)--(1.5,-.5)--(1.5,-1.5)--(1,-1)--(1,-.5)node[right]{};
\draw[thick](.5,-1)--(1,-1)--(1,-2)--(.5,-1.5)--(.5,-1)node[right]{};
\draw[thick](0,-1.5)--(.5,-1.5)--(.5,-2.5)--(0,-2)--(0,-1.5)node[right]{};
\draw[dashed](-2,-3.5)--(3,1.5)node[right]{$y-x =c$};
\end{tikzpicture}
\end{center}

The big dots corresponds to the $\varphi^1$ elements while the smaller dots are $\varphi^2$'s. For a given cocycle $\sigma$ let $Dia(\sigma) \subset \mathbb Z^2 \oplus \mathbb Z^2 \oplus \mathbb Z^2$ be the associated diagram \cite[Page 334]{Q1}. Let $\pi_i$ denote the $i^{th}$ projection map on the lattice $\mathbb Z^2$. With out loss of generality assume that in the projected diagram $\pi_1(Dia(\sigma))_{0,0}$ the element $\varphi^2_{0,0} = 0$ and $\varphi^1_{0,0} \neq 0$. For each $c \in \mathbb Z$ we shall construct below $\gamma_c \in \mathcal A_\theta^{alg \ast}$ such that $\pi_1(Dia({}_{\omega}\alpha_1(\gamma_c)))_{n,m} = \pi_1(Dia(\sigma))_{n,m}$ for all $n,m \in \mathbb Z$ such that $m-n = 3c$. Thereafter, we shall prove that for $\gamma := {\sum}_{c \in \mathbb Z} \gamma_c$
$$\pi_1(Dia({}_{\omega}\alpha_1(\gamma))_{r,s} - \pi_1(Dia(\sigma))_{r,s} = (0,\bullet) \text{ for all } r,s \in \mathbb Z.$$
Hence by repetition of the arguments for the projected diagrams $\pi_2(Dia(\sigma))_{0,0}$ and $\pi_3(Dia(\sigma))_{0,0}$, we shall conclude that every cocycle is the trivial cocycle.

\begin{lemma}
For $\sigma \in ker{}_{\omega}\alpha_2$. There exists $\gamma \in \mathcal A_\theta^{alg \ast}$ such that $\pi_1(Dia({}_{\omega}\alpha_1(\gamma)))_{n,m} = \pi_1(Dia(\sigma))_{n,m}$.
\end{lemma} 

\begin{proof}
Without loss of generality we assume that $\pi_1(Dia(\sigma))_{0,0} = (\bullet,0)$.  Let $\gamma_0$ be defined as
\begin{center}
$(\gamma_0)_{n,m} = \begin{cases}
-\varphi^1_{0,0} & \text{ for } (n,m) = (-1,0)\\
0  & \text{ for } (n,m)=(0,1) \text{ and for }(n,m) \text{ such that } m-n \neq 1 . \end{cases}$\\
\end{center}
Thereafter we define $(\gamma_0)_{n,m}$ for $n<0  \text{ and }m-n=1$ in the following iterated way.
$$(\gamma_0)_{n,m} = \lambda^{-m} \varphi^1{'}_{n+1,m}.$$
Where $\varphi^1{'}_{n+1,m}$ is first enrty of $\pi_1(Dia(\sigma- {}_{\omega}\alpha_1(\gamma_0)))_{n+1,m}$.
We see that $$\pi_1(Dia(\sigma- {}_{\omega}\alpha_1(\gamma_0)))_{n,m}=(0,0) \text{ for all }(n,m) \text{ such that }m=n\leq0.$$
Similarly we can define $(\gamma_0)_{n,m}$ for $n >0$ and $m-n = 1$ in the following way:
$$(\gamma_0)_{n,m}= \lambda^{n} \varphi^1{'}_{n,m-1}.$$
With this we have defined $\gamma_0$ at all lattice points on the plane. And we can easily see that: 
$$\pi_1(Dia(\sigma- {}_{\omega}\alpha_1(\gamma_0)))_{n,m}=(0,0) \text{ for all }(n,m) \text{ such that }m=n.$$
Putting a similar argument for any given $c \in \mathbb Z$, we get $\gamma_c$ such that:
$$\pi_1(Dia(\sigma- {}_{\omega}\alpha_1(\gamma_c)))_{n,m}=(0,0) \text{ for all }(n,m) \text{ such that }m-n=3c.$$
Furthermore the element $\gamma := {\sum}_{c \in \mathbb Z} \gamma_c$ satisfies the property that:
$$\pi_1(Dia(\sigma- {}_{\omega}\alpha_1(\gamma)))_{n,m}=(0,\bullet) \text{ for all }(n,m).$$
But, all $\varphi^1$ elements zero in the diagram $\pi_1(Dia(\sigma-{}_{\omega}\alpha_1(\gamma)))$. It can be easily shown that the $\varphi^2$ elements constitute a zero cocycle.
\end{proof}
\begin{thm}\label{thm:hoch13}
$H^1(\mathcal A_\theta^{alg} \rtimes \mathbb Z_3, (\mathcal A_\theta^{alg} \rtimes \mathbb Z_3)^{\ast}) = 0$.
\end{thm}

\begin{proof}
Using the paracyclic decomposition of the group $H^1(\mathcal A_\theta^{alg} \rtimes \mathbb Z_3, (\mathcal A_\theta^{alg} \rtimes \mathbb Z_3)^ \ast)$  we have the following relation
\begin{center}
$H^1(\mathcal A_\theta^{alg} \rtimes \mathbb Z_3, (\mathcal A_\theta^{alg} \rtimes \mathbb Z_3)^ \ast) = H^1(\mathcal A_\theta, {}_{\omega}\mathcal A_\theta^{alg \ast})^{\mathbb Z_3} \displaystyle \oplus  H^1(\mathcal A_\theta, {}_{\omega^2}\mathcal A_\theta^{alg \ast})^{\mathbb Z_3} \displaystyle \oplus  H^1(\mathcal A_\theta, \mathcal A_\theta^{alg \ast})^{\mathbb Z_3} = 0 \displaystyle \oplus 0 \displaystyle \oplus 0 = 0$.
\end{center}
\end{proof}

\centerline{\uline{The case $\Gamma = \mathbb Z_4$.}}
Like the previous case, we associate to a given $i$ twisted cocycle a diagram and thereafter prove that the cocycle is trivial. A typical portion of the infinite graph resembles the one below.
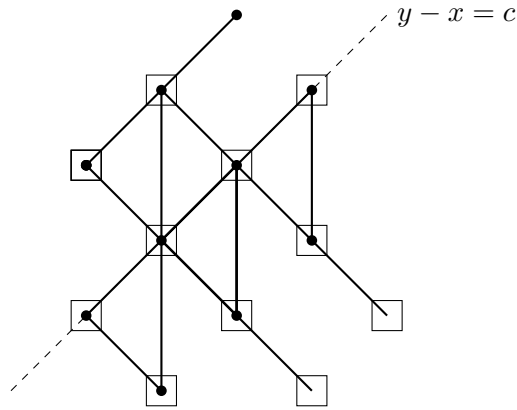
\begin{center}
\begin{tikzpicture}

\node at (-0.5,0.5) [transition]{};
\fill (-.5,0.5) circle (2pt);
\node at (0.5,1.5) [transition]{};
\fill (.5,1.5) circle (2pt);
\node at (0.5,-0.5) [transition]{};
\fill (.5,-0.5) circle (2pt);
\node at (-1.5,1.5) [transition]{};
\fill (-1.5,1.5) circle (2pt);
\node at (1.5,2.5) [transition]{};
\fill (1.5,2.5) circle (2pt);
\node at (1.5,.5) [transition]{};
\fill (1.5,.5) circle (2pt);
\node at (-.5,2.5) [transition]{};
\fill (-.5,2.5) circle (2pt);
\node at (-.5,-1.5) [transition]{};
\fill (-.5,-1.5) circle (2pt);
\node at (-1.5,-.5) [transition]{};
\fill (-1.5,-.5) circle (2pt);
\node at (-1.5,1.5) [transition]{};
\fill (-1.5,1.5) circle (2pt);
\node at (-1.5,1.5) [transition]{};
\fill (-1.5,1.5) circle (2pt);
\fill (.5,3.5) circle (2pt);
\node at (2.5,-.5) [transition]{};
\node at (1.5,-1.5) [transition]{};
\draw[thick](-.5,2.5)--(.5,3.5)node[right]{};
\draw[thick](-.5,2.5)--(.5,1.5)node[right]{};
\draw[thick](1.5,.5)--(2.5,-.5)node[right]{};
\draw[thick](.5,-.5)--(1.5,-1.5)node[right]{};
\draw[thick](-.5,.5)--(.5,1.5)--(.5,-0.5)--(-.5,.5)node[right]{};
\draw[thick](.5,1.5)--(1.5,2.5)--(1.5,0.5)--(.5,1.5)node[right]{};
\draw[thick](-1.5,-.5)--(-.5,.5)--(-.5,-1.5)--(-1.5,-.5)node[right]{};
\draw[thick](-1.5,1.5)--(-.5,2.5)--(-.5,.5)--(-1.5,1.5)node[right]{};
\draw[thick](-.5,.5)--(.5,1.5)--(.5,-0.5)--(-.5,.5)node[right]{};
\draw[thick](-.5,.5)--(.5,1.5)--(.5,-0.5)--(-.5,.5)node[right]{};
\draw[dashed](-2.5,-1.5)--(2.5,3.5)node[right]{$y-x =c$};

\end{tikzpicture}
\end{center}
The squares corresponds to the $\varphi^1$ elements while the filled dots are $\varphi^2$'s. The diagram of a given cocycle $\sigma$, $Dia(\sigma) \subset \mathbb Z^2 \oplus \mathbb Z^2$. With out loss of generality assume that in the projected diagram $\pi_1(Dia(\sigma))_{0,0}$ the element $\varphi^2_{0,0}, \varphi^1_{0,0} \neq 0$

\begin{lemma}
For $\sigma \in ker{}_{i}\alpha_2$. There exists $\gamma \in \mathcal A_\theta^{alg \ast}$ such that $\pi_1(Dia({}_{i}\alpha_1(\gamma)))_{n,m} = \pi_1(Dia(\sigma))_{n,m}$.
\end{lemma} 

\begin{proof}
 Let $\gamma_0$ be defined as
\begin{center}
$(\gamma_0)_{n,m} = \begin{cases}
-\varphi^1_{0,0} & \text{ for } (n,m) = (-1,0)\\
0  & \text{ for } (n,m)=(0,1) \text{ and for }(n,m) \text{ such that } m-n \neq 1 . \end{cases}$\\
\end{center}
Thereafter we define $(\gamma_0)_{n,m}$ for $n<0  \text{ and }m-n=1$ in the following iterated way.
$$(\gamma_0)_{n,m} = -\lambda^{-m} \varphi^1{'}_{n+1,m}.$$
Where $\varphi^1{'}_{n+1,m}$ is first enrty of $\pi_1(Dia(\sigma- {}_{i}\alpha_1(\gamma_0)))_{n+1,m}$.
We see that $$\pi_1(Dia(\sigma- {}_{\omega}\alpha_1(\gamma_0)))_{n,m}=(0,0) \text{ for all }(n,m) \text{ such that }m=n\leq0.$$
Similarly we can define $(\gamma_0)_{n,m}$ for $n >0$ and $m-n = 1$ in the following way:
$$(\gamma_0)_{n,m}= \lambda^{n} \varphi^1{'}_{n,m-1}.$$
With this we have defined $\gamma_0$ at all lattice points on the plane. And similar computations gives us that for $c \in \mathbb Z$ : 
$$\pi_1(Dia(\sigma- {}_{i}\alpha_1(\gamma_c)))_{n,m}=(0,0) \text{ for all }(n,m) \text{ such that }m-n=2c.$$
Finally the element $\gamma := {\sum}_{c \in \mathbb Z} \gamma_c$ satisfies the property that:
$$\pi_1(Dia(\sigma- {}_{i}\alpha_1(\gamma)))_{n,m}=(0,\bullet) \text{ for all }(n,m).$$
Since, all $\varphi^1$ elements zero in the diagram $\pi_1(Dia(\sigma-{}_{\omega}\alpha_1(\gamma)))$, it can be easily shown that the $\varphi^2$ elements constitute a zero cocycle.

\end{proof}
 \begin{thm}\label{thm:hoch14}
$H^1(\mathcal A_\theta^{alg} \rtimes \mathbb Z_4, (\mathcal A_\theta^{alg} \rtimes \mathbb Z_4)^{\ast}) = 0$.
\end{thm}

\begin{proof}
Using the previous lemma and the paracyclic decomposition of the group $H^1(\mathcal A_\theta^{alg} \rtimes \mathbb Z_4, (\mathcal A_\theta^{alg} \rtimes \mathbb Z_4)^ \ast)$  we have the following relation
\begin{center}
$H^1(\mathcal A_\theta^{alg} \rtimes \mathbb Z_4, (\mathcal A_\theta^{alg} \rtimes \mathbb Z_4)^ \ast) = H^1(\mathcal A_\theta, {}_{i}\mathcal A_\theta^{alg \ast})^{\mathbb Z_4} \displaystyle \oplus  H^1(\mathcal A_\theta, {}_{-i}\mathcal A_\theta^{alg \ast})^{\mathbb Z_4}\displaystyle \oplus  H^1(\mathcal A_\theta, {}_{-1}\mathcal A_\theta^{alg \ast})^{\mathbb Z_4} \displaystyle \oplus  H^1(\mathcal A_\theta,\mathcal A_\theta^{alg \ast})^{\mathbb Z_4} = 0 \displaystyle \oplus 0 \displaystyle \oplus 0 \displaystyle \oplus 0= 0$.
\end{center}
\end{proof}
\centerline{\uline{The case $\Gamma = \mathbb Z_6$.}}
We associate to a given $-\omega$ twisted cocycle $\sigma$, a diagram $Dia(\sigma)$ and thereafter prove that the cocycle is trivial. The argument here is similar to the previous two cases.

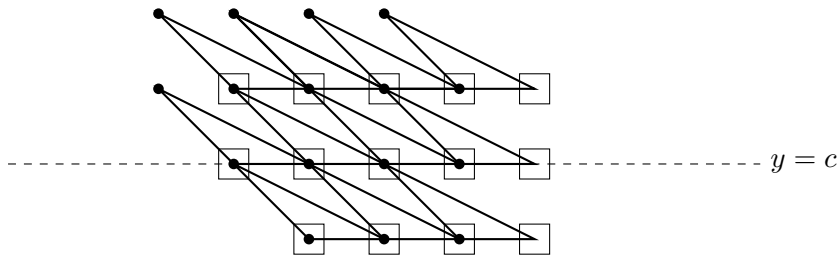
\begin{center}
\begin{tikzpicture}
\node at (-1,0) [transition]{};
\fill (-1,0) circle (2pt);
\node at (-1,-1) [transition]{};
\fill (-1,-1) circle (2pt);
\node at (0,0) [transition]{};
\fill (0,0) circle (2pt);
\fill (-1,1) circle (2pt);
\node at (-2,0) [transition]{};
\fill (-2,0) circle (2pt);
\fill (0,1) circle (2pt);
\node at (0,-1) [transition]{};
\fill (0,-1) circle (2pt);
\fill (-2,1) circle (2pt);
\fill (-3,1) circle (2pt);
\node at (1,-1) [transition]{};
\node at (2,-1) [transition]{};
\node at (1,0) [transition]{};
\node at (1,-2) [transition]{};
\node at (2,-2) [transition]{};
\node at (0,-2) [transition]{};
\node at (-1,-2) [transition]{};
\node at (-2,-1) [transition]{};

\node at (2,0) [transition]{};
\fill  (1,-2) circle(2pt);
\fill (0,-2) circle(2pt);
\fill (-3,0) circle(2pt);
\fill (1,-1) circle(2pt);
\fill (-1,-2) circle(2pt);
\fill (-2,-1) circle(2pt);
\fill (1,0) circle(2pt);
\draw[thick](1,0)--(0,0)--(-2, 1)--(-1,0)node[right]{};
\draw[thick](-2,0)--(-1,0)--(-3, 1)--(-2,0)node[right]{};
\draw[thick](0,0)--(1,0)--(-1, 1)--(0,0)node[right]{};
\draw[thick](-1,-1)--(0,-1)--(-2, 0)--(-1,-1)node[right]{};
\draw[thick](-2,-1)--(-1,-1)--(-3, 0)--(-2,-1)node[right]{};
\draw[thick](0,-1)--(1,-1)--(-1, 0)--(0,-1)node[right]{};
\draw[thick](1,-1)--(2,-1)--(0, 0)--(1,-1)node[right]{};
\draw[thick](-1,-2)--(0,-2)--(-2, -1)--(-1,-2)node[right]{};
\draw[thick](-1,0)--(0,0)--(-2, 1)--(-1,0)node[right]{};
\draw[thick](0,-2)--(1,-2)--(-1, -1)--(0,-2)node[right]{};
\draw[thick](1,-2)--(2,-2)--(0, -1)--(1,-2)node[right]{};
\draw[thick](1,0)--(2,0)--(0, 1)--(1,0)node[right]{};

\draw[dashed](-5,-1)--(5,-1)node[right]{$y=c$};
\end{tikzpicture}
\end{center}
The squares corresponds to the $\varphi^1$ elements while the filled dots are $\varphi^2$'s. For a given cocycle $\sigma$, its diagram $Dia(\sigma) \subset \mathbb Z^2$.

\begin{lemma}
For $\sigma \in ker{}_{-\omega}\alpha_2$. There exists $\gamma \in \mathcal A_\theta^{alg \ast}$ such that $Dia({}_{-\omega}\alpha_1(\gamma))_{n,m} = Dia(\sigma)_{n,m}$.
\end{lemma} 

\begin{proof}
Let $\gamma_0$ be defined as
\begin{center}
$(\gamma_0)_{n,m} = \begin{cases}
-\varphi^1_{0,0} & \text{ for } (n,m) = (-1,0)\\
0  & \text{ for } (n,m) \text{ such that } m \neq 0 . \end{cases}$\\
\end{center}
Thereafter we define $(\gamma_0)_{n,m}$ for $n<0  \text{ and }m=0$ in the following iterated way.
$$(\gamma_0)_{n,m} = -\lambda^{-m} \varphi^1{'}_{n+1,m}.$$
Where $\varphi^1{'}_{n+1,m}$ is first enrty of $Dia(\sigma- {}_{\omega}\alpha_1(\gamma_0))_{n+1,m}$.
We see that $$Dia(\sigma- {}_{\omega}\alpha_1(\gamma_0))_{n,m}=(0,0) \text{ for all }(n,m) \text{ such that }m=0.$$
Similarly we can define $(\gamma_0)_{n,m}$ for $n >0$ and $m = 0$ in the following way:
$$(\gamma_0)_{n,m}= \lambda^{-n} \varphi^1{'}_{n,m-1}.$$
With this we have defined $\gamma_0$ at all lattice points on the plane satisfying: 
$$Dia(\sigma- {}_{\omega}\alpha_1(\gamma_0))_{n,m}=(0,0) \text{ for all }(n,m) \text{ such that }m=0.$$
Similarly for $c \in \mathbb Z$, we get construct $\gamma_c$ and  the element $\gamma := {\sum}_{c \in \mathbb Z} \gamma_c$ satisfies the following property:
$$Dia(\sigma- {}_{\omega}\alpha_1(\gamma))_{n,m}=(0,\bullet) \text{ for all }(n,m).$$
But, all $\varphi^1$ elements zero in the diagram $\pi_1(Dia(\sigma-{}_{\omega}\alpha_1(\gamma)))$. It can be easily shown that the $\varphi^2$ elements constitute a zero cocycle.
\end{proof}
\begin{thm}\label{thm:hoch16}
$H^1(\mathcal A_\theta^{alg} \rtimes \mathbb Z_6, (\mathcal A_\theta^{alg} \rtimes \mathbb Z_6)^{\ast}) = 0$.
\end{thm}

\begin{proof}
Using the paracyclic decomposition of the group $H^1(\mathcal A_\theta^{alg} \rtimes \mathbb Z_6, (\mathcal A_\theta^{alg} \rtimes \mathbb Z_6)^ \ast)$  we have the following relation
\begin{center}
$H^1(\mathcal A_\theta^{alg} \rtimes \mathbb Z_6, (\mathcal A_\theta^{alg} \rtimes \mathbb Z_6)^ \ast) = H^1(\mathcal A_\theta, {}_{-\omega}\mathcal A_\theta^{alg \ast})^{\mathbb Z_6}\displaystyle \oplus  H^1(\mathcal A_\theta, {}_{-\omega^2}\mathcal A_\theta^{alg \ast})^{\mathbb Z_6}\displaystyle \oplus  H^1(\mathcal A_\theta, {}_{\omega}\mathcal A_\theta^{alg \ast})^{\mathbb Z_6}\displaystyle \oplus  H^1(\mathcal A_\theta, {}_{\omega^2}\mathcal A_\theta^{alg \ast})^{\mathbb Z_6} \displaystyle \oplus  H^1(\mathcal A_\theta, {}_{-1}\mathcal A_\theta^{alg \ast})^{\mathbb Z_6} \displaystyle \oplus  H^1(\mathcal A_\theta, \mathcal A_\theta^{alg \ast})^{\mathbb Z_6} = 0 \displaystyle \oplus 0\displaystyle \oplus 0\displaystyle \oplus 0\displaystyle \oplus 0 \displaystyle \oplus 0 = 0$.
\end{center}
\end{proof}

\begin{proof}[Proof of Theorem \ref{thm:hoch}]
The Theorems  \ref{thm:hoch03}  \ref{thm:hoch04} and  \ref{thm:hoch06} describe the zeroth Hochschild cohomology groups while the Theorems \ref{thm:hoch13}  \ref{thm:hoch14} and  \ref{thm:hoch16} resolves the dimension of the first Hochschild cohomology of these three orbifolds. The Theorems \ref{thm:hoch23}  \ref{thm:hoch24} and  \ref{thm:hoch26} state that the dimension of the second Hochschild cohomology in all the three cases is $1$.
\end{proof}

\section{Periodic Cyclic Cohomology}

\begin{center}
The case $\Gamma = \mathbb Z_3$.
\end{center}
\subsection{Cyclic cohomology groups}
\begin{lemma} \label{thm:cyc1233}
For the algebraic noncommutative toroidal orbifold $\mathcal A_\theta^{alg} \rtimes \mathbb Z_3$, we have,
$$HC^0(\mathcal A_\theta^{alg} \rtimes \mathbb Z_3) \cong \mathbb C^{7}; \; HC^1(\mathcal A_\theta^{alg} \rtimes \mathbb Z_3) \cong 0;\; HC^2(\mathcal A_\theta^{alg} \rtimes \mathbb Z_3) \cong \mathbb C^{8}.$$
\end{lemma}
\begin{proof} 
We consider the $B,S,I$ \cite[Section 2.4]{L} long exact sequence for Hochschild and cyclic cohomology.
\begin{center}
$\cdots \rightarrow H^1(\mathcal A_\theta^{alg},{}_{\omega}\mathcal A_\theta^{alg \ast})^{\mathbb Z_3} \xrightarrow{B} HC^0(\mathcal A_\theta^{alg},{}_{\omega}\mathcal A_\theta^{alg \ast})^{\mathbb Z_3} \xrightarrow{S} HC^2(\mathcal A_\theta^{alg},{}_{\omega}\mathcal A_\theta^{alg \ast})^{\mathbb Z_3} \xrightarrow{I} H^2(\mathcal A_\theta^{alg}, {}_{\omega}\mathcal A_\theta^{alg \ast})^{\mathbb Z_3} \xrightarrow{B}  HC^1(\mathcal A_\theta^{alg},{}_{\omega}\mathcal A_\theta^{alg \ast})^{\mathbb Z_3} \xrightarrow{S} \cdots$
\end{center}
In the above sequence $H^\bullet(\mathcal A_\theta^{alg}, {}_{\omega}\mathcal A_\theta^{alg \ast})^{\mathbb Z_3} = 0=HC^\bullet(\mathcal A_\theta^{alg}, {}_{\omega}\mathcal A_\theta^{alg \ast})^{\mathbb Z_3}$ for   $\bullet = 1$, hence $$HC^2(\mathcal A_\theta^{alg}, {}_{\omega}\mathcal A_\theta^{alg \ast})^{\mathbb Z_3}\cong HC^0(\mathcal A_\theta^{alg}, {}_{\omega}\mathcal A_\theta^{alg \ast})^{\mathbb Z_3} \displaystyle \oplus H^2(\mathcal A_\theta^{alg}, {}_{\omega}\mathcal A_\theta^{alg \ast})^{\mathbb Z_3}=HC^0(\mathcal A_\theta^{alg}, {}_{\omega}\mathcal A_\theta^{alg \ast})^{\mathbb Z_3} \cong \mathbb C^3.$$
Similarly we have 
$$HC^2(\mathcal A_\theta^{alg}, {}_{}\mathcal A_\theta^{alg \ast})^{\mathbb Z_3}\cong HC^0(\mathcal A_\theta^{alg}, {}_{}\mathcal A_\theta^{alg \ast})^{\mathbb Z_3} \displaystyle \oplus H^2(\mathcal A_\theta^{alg}, {}_{}\mathcal A_\theta^{alg \ast})^{\mathbb Z_3}\cong \mathbb C \displaystyle \oplus \mathbb C  = \mathbb C^2.$$ 
Hence  for the second cyclic cohomology we finally conclude that :
$$HC^2(\mathcal A_\theta^{alg} \rtimes \mathbb Z_3) \cong HC^2(\mathcal A_\theta^{alg}, {}_{}\mathcal A_\theta^{alg \ast})^{\mathbb Z_3}\displaystyle \oplus HC^2(\mathcal A_\theta^{alg}, {}_{\omega}\mathcal A_\theta^{alg \ast})^{\mathbb Z_3}\displaystyle \oplus HC^2(\mathcal A_\theta^{alg}, {}_{\omega^2}\mathcal A_\theta^{alg \ast})^{\mathbb Z_3}\cong \mathbb C^{8}.$$
It is straight forward to conclude that $HC^0(\mathcal A_\theta^{alg} \rtimes \mathbb Z_3) \cong \mathbb C^7$ while  $HC^1(\mathcal A_\theta^{alg} \rtimes \mathbb Z_3)$ is trivial.
\end{proof}

\begin{lemma} \label{thm:cyc1234}
For the $\mathbb Z_4$ orbifold, we have:
$$HC^0(\mathcal A_\theta^{alg} \rtimes \mathbb Z_4) \cong \mathbb C^{8}; \; HC^1(\mathcal A_\theta^{alg} \rtimes \mathbb Z_4) \cong 0;\; HC^2(\mathcal A_\theta^{alg} \rtimes \mathbb Z_4) \cong \mathbb C^{9}.$$
\end{lemma}
\begin{proof} 
Again we consider the $B,S,I$ long exact sequence for Hochschild and cyclic cohomology.
\begin{center}
$\cdots \rightarrow H^1(\mathcal A_\theta^{alg},{}_{i}\mathcal A_\theta^{alg \ast})^{\mathbb Z_4} \xrightarrow{B} HC^0(\mathcal A_\theta^{alg},{}_{i}\mathcal A_\theta^{alg \ast})^{\mathbb Z_4} \xrightarrow{S} HC^2(\mathcal A_\theta^{alg},{}_{i}\mathcal A_\theta^{alg \ast})^{\mathbb Z_4} \xrightarrow{I} H^2(\mathcal A_\theta^{alg}, {}_{i}\mathcal A_\theta^{alg \ast})^{\mathbb Z_4} \xrightarrow{B}  HC^1(\mathcal A_\theta^{alg},{}_{i}\mathcal A_\theta^{alg \ast})^{\mathbb Z_4} \xrightarrow{S} \cdots$
\end{center}
In the above sequence $H^\bullet(\mathcal A_\theta^{alg}, {}_{i}\mathcal A_\theta^{alg \ast})^{\mathbb Z_4} = 0=HC^\bullet(\mathcal A_\theta^{alg}, {}_{i}\mathcal A_\theta^{alg \ast})^{\mathbb Z_4}$ for   $\bullet = 1$, hence $$HC^2(\mathcal A_\theta^{alg}, {}_{i}\mathcal A_\theta^{alg \ast})^{\mathbb Z_4}\cong HC^0(\mathcal A_\theta^{alg}, {}_{i}\mathcal A_\theta^{alg \ast})^{\mathbb Z_4} \displaystyle \oplus H^2(\mathcal A_\theta^{alg}, {}_{i}\mathcal A_\theta^{alg \ast})^{\mathbb Z_4}=HC^0(\mathcal A_\theta^{alg}, {}_{i}\mathcal A_\theta^{alg \ast})^{\mathbb Z_4} \cong \mathbb C^2.$$
Similarly we have $HC^2(\mathcal A_\theta^{alg}, {}_{}\mathcal A_\theta^{alg \ast})^{\mathbb Z_4}\cong  \mathbb C^2$ and $HC^2(\mathcal A_\theta^{alg}, {}_{-1}\mathcal A_\theta^{alg \ast})^{\mathbb Z_4} \cong \mathbb C^3$. 
Hence  for the second cyclic cohomology we finally conclude that :

\begin{center}
$HC^2(\mathcal A_\theta^{alg} \rtimes \mathbb Z_4) \cong HC^2(\mathcal A_\theta^{alg}, {}_{}\mathcal A_\theta^{alg \ast})^{\mathbb Z_4}\displaystyle \oplus HC^2(\mathcal A_\theta^{alg}, {}_{i}\mathcal A_\theta^{alg \ast})^{\mathbb Z_4}\displaystyle \oplus HC^2(\mathcal A_\theta^{alg}, {}_{-i}\mathcal A_\theta^{alg \ast})^{\mathbb Z_4}\displaystyle \oplus HC^2(\mathcal A_\theta^{alg}, {}_{-1}\mathcal A_\theta^{alg \ast})^{\mathbb Z_4}\cong \mathbb C^{9}$.
\end{center}
The dimension of $HC^0(\mathcal A_\theta^{alg} \rtimes \mathbb Z_4)$ and  $HC^1(\mathcal A_\theta^{alg} \rtimes \mathbb Z_4)$ can be easily computed.
\end{proof}

\begin{lemma} \label{thm:cyc1236}
$$HC^0(\mathcal A_\theta^{alg} \rtimes \mathbb Z_6) \cong \mathbb C^{9}; \; HC^1(\mathcal A_\theta^{alg} \rtimes \mathbb Z_6) \cong 0;\; HC^2(\mathcal A_\theta^{alg} \rtimes \mathbb Z_6) \cong \mathbb C^{10}.$$
\end{lemma}
\begin{proof} 
Again we consider the $B,S,I$ long exact sequence for Hochschild and cyclic cohomology.
\begin{center}
$\cdots \rightarrow H^1(\mathcal A_\theta^{alg},{}_{-\omega}\mathcal A_\theta^{alg \ast})^{\mathbb Z_6} \xrightarrow{B} HC^0(\mathcal A_\theta^{alg},{}_{-\omega}\mathcal A_\theta^{alg \ast})^{\mathbb Z_6} \xrightarrow{S} HC^2(\mathcal A_\theta^{alg},{}_{-\omega}\mathcal A_\theta^{alg \ast})^{\mathbb Z_6} \xrightarrow{I} H^2(\mathcal A_\theta^{alg}, {}_{-\omega}\mathcal A_\theta^{alg \ast})^{\mathbb Z_6} \xrightarrow{B}  HC^1(\mathcal A_\theta^{alg},{}_{-\omega}\mathcal A_\theta^{alg \ast})^{\mathbb Z_6} \xrightarrow{S} \cdots$
\end{center}
In the above sequence $H^\bullet(\mathcal A_\theta^{alg}, {}_{-\omega}\mathcal A_\theta^{alg \ast})^{\mathbb Z_6} = 0=HC^\bullet(\mathcal A_\theta^{alg}, {}_{-\omega}\mathcal A_\theta^{alg \ast})^{\mathbb Z_6}$ for   $\bullet = 1$, hence $$HC^2(\mathcal A_\theta^{alg}, {}_{-\omega}\mathcal A_\theta^{alg \ast})^{\mathbb Z_6}\cong HC^0(\mathcal A_\theta^{alg}, {}_{-\omega}\mathcal A_\theta^{alg \ast})^{\mathbb Z_6} \displaystyle \oplus H^2(\mathcal A_\theta^{alg}, {}_{-\omega}\mathcal A_\theta^{alg \ast})^{\mathbb Z_6}=HC^0(\mathcal A_\theta^{alg}, {}_{-\omega}\mathcal A_\theta^{alg \ast})^{\mathbb Z_6} \cong \mathbb C.$$
We also have $HC^2(\mathcal A_\theta^{alg}, {}_{}\mathcal A_\theta^{alg \ast})^{\mathbb Z_6}\cong  \mathbb C^2$, $HC^2(\mathcal A_\theta^{alg}, {}_{-1}\mathcal A_\theta^{alg \ast})^{\mathbb Z_6} \cong \mathbb C^2$ and $HC^2(\mathcal A_\theta^{alg}, {}_{\omega}\mathcal A_\theta^{alg \ast})^{\mathbb Z_6} \cong \mathbb C^2$. 
Hence we finally conclude that :

\begin{center}
$HC^2(\mathcal A_\theta^{alg} \rtimes \mathbb Z_6) \cong HC^2(\mathcal A_\theta^{alg}, {}_{}\mathcal A_\theta^{alg \ast})^{\mathbb Z_6}\displaystyle \oplus HC^2(\mathcal A_\theta^{alg}, {}_{-\omega}\mathcal A_\theta^{alg \ast})^{\mathbb Z_6}\displaystyle \oplus HC^2(\mathcal A_\theta^{alg}, {}_{-\omega^2}\mathcal A_\theta^{alg \ast})^{\mathbb Z_6}\displaystyle \oplus HC^2(\mathcal A_\theta^{alg}, {}_{\omega^2}\mathcal A_\theta^{alg \ast})^{\mathbb Z_6}\displaystyle \oplus HC^2(\mathcal A_\theta^{alg}, {}_{\omega}\mathcal A_\theta^{alg \ast})^{\mathbb Z_6}\displaystyle \oplus HC^2(\mathcal A_\theta^{alg}, {}_{-1}\mathcal A_\theta^{alg \ast})^{\mathbb Z_6}\cong \mathbb C^{10}$.
\end{center}
The dimension of $HC^0(\mathcal A_\theta^{alg} \rtimes \mathbb Z_4)$ and  $HC^1(\mathcal A_\theta^{alg} \rtimes \mathbb Z_4)$ can be easily computed.
\end{proof}

\begin{proof}[Proof of Theorem \ref{thm:cyclic}]
From the modified Connes complex we have $H^\bullet(\mathcal A_\theta^{alg} \rtimes \G, (\mathcal A_\theta^{alg} \rtimes \G)^\ast) = 0$ for $\G = \mathbb Z_3, \mathbb Z_4$ and $\mathbb Z_6$ and $\bullet \geq 3$. We also have the isomorphism $HC^\bullet(\mathcal A_\theta^{alg} \rtimes \G, (\mathcal A_\theta^{alg} \rtimes \G)^\ast) \cong HC^{\bullet+2}(\mathcal A_\theta^{alg} \rtimes \mathbb Z_2, (\mathcal A_\theta^{alg} \rtimes \mathbb Z_2)^\ast)$ for $\bullet > 1$. Now, using the results of Lemma \ref{thm:cyc1233} we arrive at the desired results:
$$HP^{even}(\mathcal A_\theta^{alg} \rtimes \mathbb Z_3) \cong \mathbb C^8 \text{ and }HP^{odd}(\mathcal A_\theta^{alg} \rtimes \mathbb Z_3) = 0.$$
Similarly the Lemmas \ref{thm:cyc1234} and \ref{thm:cyc1236} yields the following:
$$HP^{even}(\mathcal A_\theta^{alg} \rtimes \mathbb Z_4) \cong \mathbb C^9 \text{ and }HP^{odd}(\mathcal A_\theta^{alg} \rtimes \mathbb Z_4) = 0.$$
 and 
$$HP^{even}(\mathcal A_\theta^{alg} \rtimes \mathbb Z_6) \cong \mathbb C^{10} \text{ and }HP^{odd}(\mathcal A_\theta^{alg} \rtimes \mathbb Z_6) = 0.$$
This completes the proof.
\end{proof}

\section{Chern-Connes Index}
The projections of the noncommutative smooth torus orbifolds, $\mathcal A_\theta \rtimes \mathbb Z_3$, $\mathcal A_\theta \rtimes \mathbb Z_4$ and $\mathcal A_\theta \rtimes \mathbb Z_6$ were calculated in \cite{ELPH}. In each of these cases all except one projection is algebraic and hence an element of the respective $K_0$ group. In this section we shall pair these projections with the above calculated periodic even cohomology cocycles and hence obtain a table with the Chern-Connes index for the noncommutative algebraic orbifolds.
\centerline{\uline{The case $\Gamma = \mathbb Z_3$.}}
Let $\zeta = e^{\frac{ \pi i }{3}}$ and $t$ satisfy the relations $t^3 =1$, $ t U_1 t^{-1} =  \frac{U_1^{-1} U_2}{\sqrt \lambda}$ and $t U_2 t^{-1} = U_1^{-1}$. The known projections of $\mathcal A_\theta^{alg} \rtimes \mathbb Z_3$ are as follows:
\begin{enumerate}
\item[(i)]$[1]$
\item[(ii)]$[p^{\theta}_0]$, where $p^{\theta}_0= \displaystyle \frac{1}{3}(1+t+t^2)$.
\item[(ii)]$[p^{\theta}_1]$, where $p^{\theta}_1= \displaystyle \frac{1}{3}(1+\zeta^2 t+(\omega t)^2)$.
\item[(iii)]$[q_0^{\theta}]$, where $q_0^{\theta}= \displaystyle \frac{1}{3}(1+e^{(\frac{2\pi i(2+\theta)}{6})}( U_1 t)+(e^{(2 \pi i\frac{2+\theta}{3})}( U_1 t)^2))$.
\item[(iv)]$[q_1^{\theta}]$, where $q_1^{\theta}=   \displaystyle \frac{1}{3}(1+e^{(\frac{2\pi i(2+\theta)}{6})}( \zeta^2 U_1 t)+(e^{(2 \pi i\frac{2+\theta}{3})}(\zeta^2 U_1 t)^2))$.
\item[(v)]$[r_0^{\theta}]$, where $r_0^{\theta}= \displaystyle \frac{1}{3}(1+U_1^2 t + (U_1^2t)^2)$.
\item[(vi)]$[r_1^{\theta}]$, where $r_1^{\theta}= \displaystyle \frac{1}{3}(1+\zeta^2 U_1^2 t + (\zeta^2 U_1^2t)^2)$.
\end{enumerate}

\begin{proof}[Proof of Theorem \ref{thm:table}](a)
The following are the Chern-Connes indices for $\mathcal A_\theta^{alg} \rtimes \mathbb Z_3$.\\
\underline{Pairing of $[S\tau]$} \newline
1. $\langle [1],[S\tau] \rangle = 1$ \\
2. $\langle [p_0^\theta],[S\tau] \rangle = \displaystyle \frac{1}{3}$ \\
3. $\langle [p_1^\theta],[S\tau] \rangle = \displaystyle \frac{1}{3}$ \\
4. $\langle [q_0^\theta],[S\tau] \rangle = \displaystyle \frac{1}{3}$ \\
5. $\langle [q_1^\theta],[S\tau] \rangle = \displaystyle \frac{1}{3}$. \\
6. $\langle [r_0^\theta],[S\tau] \rangle = \displaystyle \frac{1}{3}$. \\
7. $\langle [r_1^\theta],[S\tau] \rangle = \displaystyle \frac{1}{3}$. \\

\underline{Pairing of $[S\mathcal E_{0,0}^\omega]$} \newline

1. $\langle  [1],[S\mathcal E_{0,0}^\omega] \rangle = 0$ \\
2. $\langle  [p_0^\theta],[S\mathcal E_{0,0}^\omega] \rangle = \displaystyle \frac{1}{3}$ \\
3. $\langle  [p_1^\theta],[S\mathcal E_{0,0}^\omega] \rangle = \displaystyle \frac{\zeta^2}{3}$ \\
4. $\langle  [q_0^\theta],[S\mathcal E_{0,0}^\omega] \rangle = 0$ \\
5. $\langle  [q_1^\theta],[S\mathcal E_{0,0}^\omega] \rangle = 0$ \\
6. $\langle  [r_0^\theta],[S\mathcal E_{0,0}^\omega] \rangle = 0$ \\
7. $\langle  [r_1^\theta],[S\mathcal E_{0,0}^\omega] \rangle = 0$. \\

\underline{Pairing of $[S\mathcal E_{0,0}^{\omega^2}]$} \newline
1. $\langle  [1],[S\mathcal E_{0,0}^{\omega^2}] \rangle = 0$ \\
2. $\langle  [p_0^\theta],[S\mathcal E_{0,0}^{\omega^2}] \rangle = \displaystyle \frac{1}{3}$ \\
3. $\langle  [p_1^\theta],[S\mathcal E_{0,0}^{\omega^2}] \rangle = \displaystyle \frac{\zeta^4}{3}$ \\
4. $\langle  [q_0^\theta],[S\mathcal E_{0,0}^{\omega^2}] \rangle = 0$ \\
5. $\langle  [q_1^\theta],[S\mathcal E_{0,0}^{\omega^2}] \rangle = 0$. \\
6. $\langle  [r_0^\theta],[S\mathcal E_{0,0}^{\omega^2}] \rangle = 0$ \\
7. $\langle  [r_1^\theta],[S\mathcal E_{0,0}^{\omega^2}] \rangle = 0$. \\

\underline{Pairing of $[S\mathcal E_{0,1}^{\omega}]$} \newline
1. $\langle  [1],[S\mathcal E_{0,1}^{\omega}] \rangle = 0$ \\
2. $\langle  [p_0^\theta],[S\mathcal E_{0,1}^{\omega}] \rangle = 0$ \\
3. $\langle  [p_1^\theta],[S\mathcal E_{0,1}^{\omega}] \rangle = 0$ \\
4. $\langle  [q_0^\theta],[S\mathcal E_{0,1}^{\omega}] \rangle =0$ \\
5. $\langle  [q_1^\theta],[S\mathcal E_{0,1}^{\omega}] \rangle = 0$. \\
6. $\langle  [r_0^\theta],[S\mathcal E_{0,1}^{\omega}] \rangle = 0$ \\
7. $\langle  [r_1^\theta],[S\mathcal E_{0,1}^{\omega}] \rangle = 0$. \\

\underline{Pairing of $[S\mathcal E_{0,1}^{\omega^2}]$} \newline
1. $\langle  [1],[S\mathcal E_{0,1}^{\omega^2}] \rangle = 0$ \\
2. $\langle  [p_0^\theta],[S\mathcal E_{0,1}^{\omega^2}] \rangle = 0$ \\
3. $\langle  [p_1^\theta],[S\mathcal E_{0,1}^{\omega^2}] \rangle = 0$ \\
4. $\langle  [q_0^\theta],[S\mathcal E_{0,1}^{\omega^2}] \rangle = \displaystyle \frac{\zeta^2}{3 \sqrt[3]{\lambda^2}}$ \\
5. $\langle  [q_1^\theta],[S\mathcal E_{0,1}^{\omega^2}] \rangle = \displaystyle \frac{1}{3 \sqrt[3]{\lambda^2}}$. \\
6. $\langle  [r_0^\theta],[S\mathcal E_{0,1}^{\omega^2}] \rangle = 0$ \\
7. $\langle  [r_1^\theta],[S\mathcal E_{0,1}^{\omega^2}] \rangle = 0$. \\

\underline{Pairing of $[S\mathcal E_{0,-1}^{\omega}]$} \newline
1. $\langle  [1],[S\mathcal E_{0,-1}^{\omega}] \rangle = 0$ \\
2. $\langle  [p_0^\theta],[S\mathcal E_{0,-1}^{\omega}] \rangle = 0$ \\
3. $\langle  [p_1^\theta],[S\mathcal E_{0,-1}^{\omega}] \rangle = 0$ \\
4. $\langle  [q_0^\theta],[S\mathcal E_{0,-1}^{\omega}] \rangle = \displaystyle \frac{\zeta \sqrt[6]{\lambda}}{3}$ \\
5. $\langle  [q_1^\theta],[S\mathcal E_{0,-1}^{\omega}] \rangle = \displaystyle \frac{- \sqrt[6]{\lambda}}{3}$. \\
6. $\langle  [r_0^\theta],[S\mathcal E_{0,-1}^{\omega}] \rangle = 0$ \\
7. $\langle  [r_1^\theta],[S\mathcal E_{0,-1}^{\omega}] \rangle = 0$. \\

\underline{Pairing of $[S\mathcal E_{0,-1}^{\omega^2}]$} \newline
1. $\langle  [1],[S\mathcal E_{0,-1}^{\omega^2}] \rangle = 0$ \\
2. $\langle  [p_0^\theta],[S\mathcal E_{0,-1}^{\omega^2}] \rangle = 0$ \\
3. $\langle  [p_1^\theta],[S\mathcal E_{0,-1}^{\omega^2}] \rangle = 0$ \\
4. $\langle  [q_0^\theta],[S\mathcal E_{0,-1}^{\omega^2}] \rangle =0$ \\
5. $\langle  [q_1^\theta],[S\mathcal E_{0,-1}^{\omega^2}] \rangle = 0$. \\
6. $\langle  [r_0^\theta],[S\mathcal E_{0,-1}^{\omega^2}] \rangle = 0$ \\
7. $\langle  [r_1^\theta],[S\mathcal E_{0,-1}^{\omega^2}] \rangle = 0$. \\

\underline{Pairing of $[S\varphi]$} \newline
1. $\langle [1],[S\varphi] \rangle = 0$ \\
2. $\langle [p_0^\theta],[S\varphi] \rangle = 0$ \\
3. $\langle [p_1^\theta],[S\varphi] \rangle = 0$ \\
4. $\langle [q_0^\theta],[S\varphi] \rangle = 0$ \\
5. $\langle [q_1^\theta],[S\varphi] \rangle = 0$. \\
6. $\langle [r_0^\theta],[S\varphi] \rangle = 0$ \\
7. $\langle [r_1^\theta],[S\varphi] \rangle = 0$ \\
\end{proof}

\centerline{\uline{The case $\Gamma = \mathbb Z_4$.}}
Let $t$ satisfy the relations $t^4 =1$, $ t U_1 t^{-1} =  U_2$ and $t U_2 t^{-1} = U_1^{-1}$. The known projections of $\mathcal A_\theta^{alg} \rtimes \mathbb Z_4$ are as follows:
\begin{enumerate}
\item[(i)]$[1]$
\item[(ii)]$[p^{\theta}_0]$, where $p^{\theta}_0= \displaystyle \frac{1}{4}(1+t+t^2+t
^3)$.
\item[(iii)]$[p^{\theta}_1]$, where $p^{\theta}_1= \displaystyle \frac{1}{4}(1+ i t -t^2-it^3)$.
\item[(iv)]$[p^{\theta}_2]$, where $p^{\theta}_2= \displaystyle \frac{1}{4}(1-t+t^2-t^3)$.
\item[(v)]$[q_0^{\theta}]$, where $q_0^{\theta}= \displaystyle \frac{1}{4}(1+i \sqrt[4]{\lambda}U_1 t - [\sqrt[4]{\lambda} U_1 t]^2 - i[\sqrt[4]{\lambda} U_1 t]^3)$.
\item[(vi)]$[q_1^{\theta}]$, where $q_1^{\theta}= \displaystyle \frac{1}{4}(1- \sqrt[4]{\lambda}U_1 t + [\sqrt[4]{\lambda} U_1 t]^2 - [\sqrt[4]{\lambda} U_1 t]^3)$.
\item[(vii)]$[q_2^{\theta}]$, where $q_2^{\theta}= \displaystyle \frac{1}{4}(1-i \sqrt[4]{\lambda}U_1 t - [\sqrt[4]{\lambda} U_1 t]^2 +i[\sqrt[4]{\lambda} U_1 t]^3)$.
\item[(viii)]$[r_0^{\theta}]$, where $r_0^{\theta}= \displaystyle \frac{1}{2}(1-U_1 t^2)$.
\end{enumerate}

\begin{proof}[Proof of Theorem \ref{thm:table}](b)
The following are the Chern-Connes indices for $\mathcal A_\theta^{alg} \rtimes \mathbb Z_4$.\\
\underline{Pairing of $[S\tau]$} \newline
1. $\langle [1],[S\tau] \rangle = 1$ \\
2. $\langle [p_0^\theta],[S\tau] \rangle = \displaystyle \frac{1}{4}$ \\
3. $\langle [p_1^\theta],[S\tau] \rangle = \displaystyle \frac{1}{4}$ \\
4. $\langle [p_2^\theta],[S\tau] \rangle = \displaystyle \frac{1}{4}$ \\
5. $\langle [q_0^\theta],[S\tau] \rangle = \displaystyle \frac{1}{4}$. \\
6. $\langle [q_1^\theta],[S\tau] \rangle = \displaystyle \frac{1}{4}$. \\
7. $\langle [q_2^\theta],[S\tau] \rangle = \displaystyle \frac{1}{4}$. \\
8. $\langle [r_0^\theta],[S\tau] \rangle = \displaystyle \frac{1}{2}$. \\

\underline{Pairing of $[S\mathcal D_{1,1}]$} \newline

1. $\langle [1],[S\mathcal D_{1,1}] \rangle = 0$ \\
2. $\langle [p_0^\theta],[S\mathcal D_{1,1}] \rangle = 0$ \\
3. $\langle [p_1^\theta],[S\mathcal D_{1,1}] \rangle = 0$ \\
4. $\langle [p_2^\theta],[S\mathcal D_{1,1}] \rangle = 0$ \\
5. $\langle [q_0^\theta],[S\mathcal D_{1,1}] \rangle = \displaystyle \frac{-1}{4\sqrt\lambda}$. \\
6. $\langle [q_1^\theta],[S\mathcal D_{1,1}] \rangle = \displaystyle \frac{1}{4\sqrt \lambda}$. \\
7. $\langle [q_2^\theta],[S\mathcal D_{1,1}] \rangle = \displaystyle \frac{-1}{4 \sqrt \lambda}$. \\
8. $\langle [r_0^\theta],[S\mathcal D_{1,1}] \rangle = 0$. \\

\underline{Pairing of $[S\mathcal D_{0,0}]$} \newline
1. $\langle [1],[S\mathcal D_{0,0}] \rangle = 0$ \\
2. $\langle [p_0^\theta],[S\mathcal D_{0,0}] \rangle = \displaystyle \frac{1}{4}$ \\
3. $\langle [p_1^\theta],[S\mathcal D_{0,0}] \rangle = \displaystyle \frac{-1}{4}$ \\
4. $\langle [p_2^\theta],[S\mathcal D_{0,0}] \rangle = \displaystyle \frac{1}{4}$ \\
5. $\langle [q_0^\theta],[S\mathcal D_{0,0}] \rangle = 0$. \\
6. $\langle [q_1^\theta],[S\mathcal D_{0,0}] \rangle = 0$. \\
7. $\langle [q_2^\theta],[S\mathcal D_{0,0}] \rangle = 0$. \\
8. $\langle [r_0^\theta],[S\mathcal D_{0,0}] \rangle = 0$. \\

\underline{Pairing of $[S(\mathcal D_{0,1}+\mathcal D_{1,0})]$} \newline
1. $\langle [1],[S(\mathcal D_{0,1}+\mathcal D_{1,0})] \rangle = 0$ \\
2. $\langle [p_0^\theta],[S(\mathcal D_{0,1}+\mathcal D_{1,0})] \rangle = 0$ \\
3. $\langle [p_1^\theta],[S(\mathcal D_{0,1}+\mathcal D_{1,0})] \rangle = 0$ \\
4. $\langle [p_2^\theta],[S(\mathcal D_{0,1}+\mathcal D_{1,0})] \rangle = 0$ \\
5. $\langle [q_0^\theta],[S(\mathcal D_{0,1}+\mathcal D_{1,0})] \rangle = 0$. \\
6. $\langle [q_1^\theta],[S(\mathcal D_{0,1}+\mathcal D_{1,0})] \rangle = 0$. \\
7. $\langle [q_2^\theta],[S(\mathcal D_{0,1}+\mathcal D_{1,0})] \rangle = 0$. \\
8. $\langle [r_0^\theta],[S(\mathcal D_{0,1}+\mathcal D_{1,0})] \rangle = \displaystyle \frac{-1}{2}$. \\

\underline{Pairing of $[S\mathcal F^i_{0,0}]$} \newline
1. $\langle [1],[S\mathcal F^i_{0,0}] \rangle = 0$ \\
2. $\langle [p_0^\theta],[S\mathcal F^i_{0,0}] \rangle = \displaystyle \frac{1}{4}$ \\
3. $\langle [p_1^\theta],[S\mathcal F^i_{0,0}] \rangle = \displaystyle \frac{i}{4}$ \\
4. $\langle [p_2^\theta],[S\mathcal F^i_{0,0}] \rangle = \displaystyle \frac{-1}{4}$ \\
5. $\langle [q_0^\theta],[S\mathcal F^i_{0,0}] \rangle = 0$. \\
6. $\langle [q_1^\theta],[S\mathcal F^i_{0,0}] \rangle = 0$. \\
7. $\langle [q_2^\theta],[S\mathcal F^i_{0,0}] \rangle = 0$. \\
8. $\langle [r_0^\theta],[S\mathcal F^i_{0,0}] \rangle = 0$. \\

\underline{Pairing of $[S\mathcal F^i_{0,1}]$} \newline
1. $\langle [1],[S\mathcal F^i_{0,1}] \rangle = 0$ \\
2. $\langle [p_0^\theta],[S\mathcal F^i_{0,1}] \rangle = 0$ \\
3. $\langle [p_1^\theta],[S\mathcal F^i_{0,1}] \rangle = 0$ \\
4. $\langle [p_2^\theta],[S\mathcal F^i_{0,1}] \rangle = 0$ \\
5. $\langle [q_0^\theta],[S\mathcal F^i_{0,1}] \rangle = \displaystyle \frac{i \sqrt[4]{\lambda}}{4}$. \\
6. $\langle [q_1^\theta],[S\mathcal F^i_{0,1}] \rangle = \displaystyle \frac{- \sqrt[4]{\lambda}}{4}$. \\
7. $\langle [q_2^\theta],[S\mathcal F^i_{0,1}] \rangle = \displaystyle \frac{-i \sqrt[4]{\lambda}}{4}$. \\
8. $\langle [r_0^\theta],[S\mathcal F^i_{0,1}] \rangle = 0$. \\

\underline{Pairing of $[S\mathcal F^{-i}_{0,0}]$} \newline
1. $\langle [1],[S\mathcal F^{-i}_{0,0}] \rangle = 0$ \\
2. $\langle [p_0^\theta],[S\mathcal F^{-i}_{0,0}] \rangle = \displaystyle \frac{1}{4}$ \\
3. $\langle [p_1^\theta],[S\mathcal F^{-i}_{0,0}] \rangle = \displaystyle \frac{-i}{4}$ \\
4. $\langle [p_2^\theta],[S\mathcal F^{-i}_{0,0}] \rangle = \displaystyle \frac{-1}{4}$ \\
5. $\langle [q_0^\theta],[S\mathcal F^{-i}_{0,0}] \rangle = 0$. \\
6. $\langle [q_1^\theta],[S\mathcal F^{-i}_{0,0}] \rangle = 0$. \\
7. $\langle [q_2^\theta],[S\mathcal F^{-i}_{0,0}] \rangle = 0$. \\
8. $\langle [r_0^\theta],[S\mathcal F^{-i}_{0,0}] \rangle = 0$. \\

\underline{Pairing of $[S\mathcal F^{-i}_{0,1}]$} \newline
1. $\langle [1],[S\mathcal F^{-i}_{0,1}] \rangle = 0$ \\
2. $\langle [p_0^\theta],[S\mathcal F^{-i}_{0,1}] \rangle = 0$ \\
3. $\langle [p_1^\theta],[S\mathcal F^{-i}_{0,1}] \rangle = 0$ \\
4. $\langle [p_2^\theta],[S\mathcal F^{-i}_{0,1}] \rangle = 0$ \\
5. $\langle [q_0^\theta],[S\mathcal F^{-i}_{0,1}] \rangle = \displaystyle \frac{-i}{4\sqrt[4]{\lambda}}$. \\
6. $\langle [q_1^\theta],[S\mathcal F^{-i}_{0,1}] \rangle = \displaystyle \frac{-1}{4\sqrt[4]{\lambda}}$. \\
7. $\langle [q_2^\theta],[S\mathcal F^{-i}_{0,1}] \rangle = \displaystyle \frac{i}{4\sqrt[4]{\lambda}}$. \\
8. $\langle [r_0^\theta],[S\mathcal F^{-i}_{0,1}] \rangle = 0$. \\

\underline{Pairing of $[S\varphi]$} \newline
1. $\langle [1],[S\varphi] \rangle = 0$ \\
2. $\langle [p_0^\theta],[S\varphi] \rangle = 0$ \\
3. $\langle [p_1^\theta],[S\varphi] \rangle = 0$ \\
4. $\langle [p_2^\theta],[S\varphi] \rangle = 0$ \\
5. $\langle [q_0^\theta],[S\varphi] \rangle = 0$. \\
6. $\langle [q_1^\theta],[S\varphi] \rangle = 0$. \\
7. $\langle [q_2^\theta],[S\varphi] \rangle = 0$. \\
8. $\langle [r_0^\theta],[S\varphi] \rangle = 0$. \\
\end{proof}

\centerline{\uline{The case $\Gamma = \mathbb Z_6$.}}
Let $t$ satisfy the relations $t^6 =1$, $ t U_1 t^{-1} =  U_2$ and $t U_2 t^{-1} = \frac{U_1^{-1}U_2}{\sqrt\lambda}$. The known projections of $\mathcal A_\theta^{alg} \rtimes \mathbb Z_6$ are as follows:
\begin{enumerate}
\item[(i)]$[1]$
\item[(ii)]$[p^{\theta}_0]$, where $p^{\theta}_0= \displaystyle \frac{1}{6}(1+t+t^2+t
^3+t^4+t^5)$.
\item[(iii)]$[p^{\theta}_1]$, where $p^{\theta}_1= \displaystyle \frac{1}{6}(1+\zeta t+\zeta^2 t^2- t
^3+\zeta^4 t^4+ \zeta^5 t^5)$.
\item[(iv)]$[p^{\theta}_2]$, where $p^{\theta}_2= \displaystyle \frac{1}{6}(1+\zeta^2 t+\zeta^4 t^2+ t
^3- t^4+ \zeta^4 t^5)$.
\item[(v)]$[p^{\theta}_3]$, where $p^{\theta}_3= \displaystyle \frac{1}{6}(1- t+ t^2- t
^3+ t^4- t^5)$.
\item[(vi)]$[p^{\theta}_4]$, where $p^{\theta}_4= \displaystyle \frac{1}{6}(1+\zeta^4 t+\zeta^2 t^2+ t
^3+\zeta^4 t^4+ \zeta^2 t^5)$.
\item[(vii)]$[q_0^{\theta}]$, where $q_0^{\theta}= \displaystyle \frac{1}{3}(1+e^{(\frac{2\pi i(2+\theta)}{6})}( U_1 t^2)+[e^{(2 \pi i\frac{2+\theta}{6})}( U_1 t^2)]^2)$.
\item[(viii)]$[q_1^{\theta}]$, where $q_1^{\theta}=\displaystyle \frac{1}{3}(1+\zeta^2e^{(\frac{2\pi i(2+\theta)}{6})}( U_1 t^2)+\zeta^4[e^{(2 \pi i\frac{2+\theta}{6})}( U_1 t^2)]^2)$.
\item[(ix)]$[r^{\theta}]$, where $r^{\theta}= \displaystyle \frac{1}{2}(1-U_1 t^3)$.
\end{enumerate}

\begin{proof}[Proof of Theorem \ref{thm:table}](c)
The following are the Chern-Connes indices for $\mathcal A_\theta^{alg} \rtimes \mathbb Z_6$.\\
\underline{Pairing of $[S\tau]$} \newline
1. $\langle [1],[S\tau] \rangle = 1$ \\
2. $\langle [p_0^\theta],[S\tau] \rangle = \displaystyle \frac{1}{6}$ \\
3. $\langle [p_1^\theta],[S\tau] \rangle = \displaystyle \frac{1}{6}$ \\
4. $\langle [p_2^\theta],[S\tau] \rangle = \displaystyle \frac{1}{6}$ \\
5. $\langle [q_0^\theta],[S\tau] \rangle = \displaystyle \frac{1}{6}$. \\
6. $\langle [q_1^\theta],[S\tau] \rangle = \displaystyle \frac{1}{3}$. \\
7. $\langle [q_2^\theta],[S\tau] \rangle = \displaystyle \frac{1}{3}$. \\
8. $\langle [r_0^\theta],[S\tau] \rangle = \displaystyle \frac{1}{2}$. \\

\underline{Pairing of $[S\mathcal D_{0,0}]$} \newline

1. $\langle [1],[S\mathcal D_{0,0}] \rangle = 0$ \\
2. $\langle [p_0^\theta],[S\mathcal D_{0,0}] \rangle = \displaystyle \frac{1}{6}$ \\
3. $\langle [p_1^\theta],[S\mathcal D_{0,0}] \rangle = \displaystyle \frac{-1}{6}$ \\
4. $\langle [p_2^\theta],[S\mathcal D_{0,0}] \rangle = \displaystyle \frac{1}{6}$ \\
5. $\langle [p_3^\theta],[S\mathcal D_{0,0}] \rangle = \displaystyle \frac{-1}{6}$. \\
6. $\langle [p_4^\theta],[S\mathcal D_{0,0}] \rangle = \displaystyle \frac{1}{6}$. \\
7. $\langle [q_0^\theta],[S\mathcal D_{0,0}] \rangle = 0$. \\
7. $\langle [q_1^\theta],[S\mathcal D_{0,0}] \rangle = 0$. \\
8. $\langle [r^\theta],[S\mathcal D_{0,0}] \rangle = 0$. \\

\underline{Pairing of $[S(\mathcal D_{1,0}+\lambda \sqrt{\lambda}\mathcal D_{1,0}+\sqrt\lambda \mathcal D_{1,1})]$} \newline
1. $\langle [1],[S(\mathcal D_{1,0}+\lambda \sqrt{\lambda}\mathcal D_{1,0}+\sqrt\lambda \mathcal D_{1,1})] \rangle = 0$ \\
2. $\langle [p_0^\theta],[S(\mathcal D_{1,0}+\lambda \sqrt{\lambda}\mathcal D_{1,0}+\sqrt\lambda \mathcal D_{1,1})] \rangle = 0$ \\
3. $\langle [p_1^\theta],[S(\mathcal D_{1,0}+\lambda \sqrt{\lambda}\mathcal D_{1,0}+\sqrt\lambda \mathcal D_{1,1})] \rangle = 0$ \\
4. $\langle [p_2^\theta],[S(\mathcal D_{1,0}+\lambda \sqrt{\lambda}\mathcal D_{1,0}+\sqrt\lambda \mathcal D_{1,1})] \rangle = 0$ \\
5. $\langle [p_3^\theta],[S(\mathcal D_{1,0}+\lambda \sqrt{\lambda}\mathcal D_{1,0}+\sqrt\lambda \mathcal D_{1,1})] \rangle = 0$. \\
6. $\langle [p_4^\theta],[S(\mathcal D_{1,0}+\lambda \sqrt{\lambda}\mathcal D_{1,0}+\sqrt\lambda \mathcal D_{1,1})] \rangle = 0$. \\
7. $\langle [q_0^\theta],[S(\mathcal D_{1,0}+\lambda \sqrt{\lambda}\mathcal D_{1,0}+\sqrt\lambda \mathcal D_{1,1})] \rangle = 0$. \\
8. $\langle [q_0^\theta],[S(\mathcal D_{1,0}+\lambda \sqrt{\lambda}\mathcal D_{1,0}+\sqrt\lambda \mathcal D_{1,1})] \rangle = 0$. \\
9. $\langle [r_0^\theta],[S(\mathcal D_{1,0}+\lambda \sqrt{\lambda}\mathcal D_{1,0}+\sqrt\lambda \mathcal D_{1,1})] \rangle = \displaystyle \frac{-\lambda \sqrt\lambda}{2}$. \\

\underline{Pairing of $[S(\mathcal E^\omega_{0,1}+\mathcal E^\omega_{0,-1})]$} \newline
1. $\langle [1],[S(\mathcal E^\omega_{0,1}+\mathcal E^\omega_{0,-1})] \rangle = 0$ \\
2. $\langle [p_0^\theta],[S(\mathcal E^\omega_{0,1}+\mathcal E^\omega_{0,-1})] \rangle = 0$ \\
3. $\langle [p_1^\theta],[S(\mathcal E^\omega_{0,1}+\mathcal E^\omega_{0,-1})] \rangle = 0$ \\
4. $\langle [p_2^\theta],[S(\mathcal E^\omega_{0,1}+\mathcal E^\omega_{0,-1})] \rangle = 0$ \\
5. $\langle [p_3^\theta],[S(\mathcal E^\omega_{0,1}+\mathcal E^\omega_{0,-1})] \rangle = 0$. \\
6. $\langle [p_4^\theta],[S(\mathcal E^\omega_{0,1}+\mathcal E^\omega_{0,-1})] \rangle = 0$. \\
7. $\langle [q_0^\theta],[S(\mathcal E^\omega_{0,1}+\mathcal E^\omega_{0,-1})] \rangle = \displaystyle \frac{\zeta}{3}$. \\
8. $\langle [q_1^\theta],[S(\mathcal E^\omega_{0,1}+\mathcal E^\omega_{0,-1})] \rangle = \displaystyle \frac{-1}{3}$. \\
9. $\langle [r_0^\theta],[S(\mathcal E^\omega_{0,1}+\mathcal E^\omega_{0,-1})] \rangle = 0$. \\

\underline{Pairing of $[S\mathcal E^\omega_{0,0}]$} \newline
1. $\langle [1],[S\mathcal E^\omega_{0,0}] \rangle = 0$ \\
2. $\langle [p_0^\theta],[S\mathcal E^\omega_{0,0}] \rangle = \displaystyle \frac{1}{6}$ \\
3. $\langle [p_1^\theta],[S\mathcal E^\omega_{0,0}] \rangle = \displaystyle \frac{\zeta^2}{3}$ \\
4. $\langle [p_2^\theta],[S\mathcal E^\omega_{0,0}] \rangle = \displaystyle \frac{-\zeta}{3}$ \\
5. $\langle [p_3^\theta],[S\mathcal E^\omega_{0,0}] \rangle = \displaystyle \frac{1}{6}$. \\
6. $\langle [p_4^\theta],[S\mathcal E^\omega_{0,0}] \rangle = \displaystyle \frac{\zeta^2}{6}$. \\
7. $\langle [q_0^\theta],[S\mathcal E^\omega_{0,0}] \rangle = 0$. \\
8. $\langle [q_1^\theta],[S\mathcal E^\omega_{0,0}] \rangle = 0$. \\
9. $\langle [r^\theta],[S\mathcal E^\omega_{0,0}] \rangle = 0$. \\

\underline{Pairing of $[S(\mathcal E^{\omega^2}_{0,1}+\mathcal E^{\omega^2}_{0,-1})]$} \newline
1. $\langle [1],[S(\mathcal E^{\omega^2}_{0,1}+\mathcal E^{\omega^2}_{0,-1})] \rangle = 0$ \\
2. $\langle [p_0^\theta],[S(\mathcal E^{\omega^2}_{0,1}+\mathcal E^{\omega^2}_{0,-1})] \rangle = 0$ \\
3. $\langle [p_1^\theta],[S(\mathcal E^{\omega^2}_{0,1}+\mathcal E^{\omega^2}_{0,-1})] \rangle = 0$ \\
4. $\langle [p_2^\theta],[S(\mathcal E^{\omega^2}_{0,1}+\mathcal E^{\omega^2}_{0,-1})] \rangle = 0$ \\
5. $\langle [p_3^\theta],[S(\mathcal E^{\omega^2}_{0,1}+\mathcal E^{\omega^2}_{0,-1})] \rangle = 0$. \\
6. $\langle [p_4^\theta],[S(\mathcal E^{\omega^2}_{0,1}+\mathcal E^{\omega^2}_{0,-1})] \rangle = 0$. \\
7. $\langle [q_0^\theta],[S(\mathcal E^{\omega^2}_{0,1}+\mathcal E^{\omega^2}_{0,-1})] \rangle = \displaystyle \frac{\zeta^2}{3\sqrt[6]{\lambda}}$. \\
8. $\langle [q_1^\theta],[S(\mathcal E^{\omega^2}_{0,1}+\mathcal E^{\omega^2}_{0,-1})] \rangle = \displaystyle  \frac{-\zeta}{3\sqrt[6]{\lambda}}$. \\
9. $\langle [r_0^\theta],[S(\mathcal E^{\omega^2}_{0,1}+\mathcal E^{\omega^2}_{0,-1})] \rangle = 0$. \\

\underline{Pairing of $[S\mathcal E^{\omega^2}_{0,0}]$} \newline
1. $\langle [1],[S\mathcal E^{\omega^2}_{0,0}] \rangle = 0$ \\
2. $\langle [p_0^\theta],[S\mathcal E^{\omega^2}_{0,0}] \rangle = \displaystyle \frac{1}{6}$ \\
3. $\langle [p_1^\theta],[S\mathcal E^{\omega^2}_{0,0}] \rangle = \displaystyle \frac{-\zeta}{6}$ \\
4. $\langle [p_2^\theta],[S\mathcal E^{\omega^2}_{0,0}] \rangle = \displaystyle \frac{-1}{6}$ \\
5. $\langle [p_3^\theta],[S\mathcal E^{\omega^2}_{0,0}] \rangle = \displaystyle \frac{1}{6}$. \\
6. $\langle [p_4^\theta],[S\mathcal E^{\omega^2}_{0,0}] \rangle = \displaystyle \frac{-\zeta}{6}$. \\
7. $\langle [q_0^\theta],[S\mathcal E^{\omega^2}_{0,0}] \rangle = 0$. \\
8. $\langle [q_1^\theta],[S\mathcal E^{\omega^2}_{0,0}] \rangle = 0$. \\
9. $\langle [r^\theta],[S\mathcal E^{\omega^2}_{0,0}] \rangle = 0$. \\

\underline{Pairing of $[S\mathcal G^{-\omega}_{0,0}]$} \newline
1. $\langle [1],[S\mathcal G^{-\omega}_{0,0}] \rangle = 0$ \\
2. $\langle [p_0^\theta],[S\mathcal G^{-\omega}_{0,0}] \rangle = \displaystyle \frac{1}{6}$ \\
3. $\langle [p_1^\theta],[S\mathcal G^{-\omega}_{0,0}] \rangle = \displaystyle \frac{\zeta}{6}$ \\
4. $\langle [p_2^\theta],[S\mathcal G^{-\omega}_{0,0}] \rangle = \displaystyle \frac{\zeta^2}{6}$ \\
5. $\langle [p_3^\theta],[S\mathcal G^{-\omega}_{0,0}] \rangle = \displaystyle \frac{-1}{6}$. \\
6. $\langle [p_4^\theta],[S\mathcal G^{-\omega}_{0,0}] \rangle = \displaystyle \frac{-\zeta}{6}$. \\
7. $\langle [q_0^\theta],[S\mathcal G^{-\omega}_{0,0}] \rangle = 0$. \\
8. $\langle [q_1^\theta],[S\mathcal G^{-\omega}_{0,0}] \rangle = 0$. \\
9. $\langle [r^\theta],[S\mathcal G^{-\omega}_{0,0}] \rangle = 0$. \\

\underline{Pairing of $[S\mathcal G^{-\omega^2}_{0,0}]$} \newline
1. $\langle [1],[S\mathcal G^{-\omega^2}_{0,0}] \rangle = 0$ \\
2. $\langle [p_0^\theta],[S\mathcal G^{-\omega^2}_{0,0}] \rangle = \displaystyle \frac{1}{6}$ \\
3. $\langle [p_1^\theta],[S\mathcal G^{-\omega^2}_{0,0}] \rangle = \displaystyle \frac{-\zeta^2}{6}$ \\
4. $\langle [p_2^\theta],[S\mathcal G^{-\omega^2}_{0,0}] \rangle = \displaystyle \frac{-\zeta}{6}$ \\
5. $\langle [p_3^\theta],[S\mathcal G^{-\omega^2}_{0,0}] \rangle = \displaystyle \frac{-1}{6}$. \\
6. $\langle [p_4^\theta],[S\mathcal G^{-\omega^2}_{0,0}] \rangle = \displaystyle \frac{\zeta^2}{6}$. \\
7. $\langle [q_0^\theta],[S\mathcal G^{-\omega^2}_{0,0}] \rangle = 0$. \\
8. $\langle [q_1^\theta],[S\mathcal G^{-\omega^2}_{0,0}] \rangle = 0$. \\
9. $\langle [r^\theta],[S\mathcal G^{-\omega^2}_{0,0}] \rangle = 0$. \\

\underline{Pairing of $[S\varphi]$} \newline
1. $\langle [1],[S\varphi] \rangle = 0$ \\
2. $\langle [p_0^\theta],[S\varphi] \rangle = 0$ \\
3. $\langle [p_1^\theta],[S\varphi] \rangle = 0$ \\
4. $\langle [p_2^\theta],[S\varphi] \rangle = 0$ \\
5. $\langle [p_3^\theta],[S\varphi] \rangle = 0$. \\
6. $\langle [p_4^\theta],[S\varphi] \rangle = 0$. \\
7. $\langle [q_0^\theta],[S\varphi] \rangle = 0$. \\
8. $\langle [q_1^\theta],[S\varphi] \rangle = 0$. \\
9. $\langle [r^\theta],[S\varphi] \rangle = 0$. \\
\end{proof}
\section{Conclusion and Conjectures}
The homology and cohomology groups \cite{Q1}, \cite{Q2} and computed in this article gives us a complete understanding of the noncommutative algebraic noncommutative torus orbifold. Through our computations, we see the following dualities:
$$H_\bullet(\mathcal A_\theta^{alg} \rtimes \G,\mathcal A_\theta^{alg} \rtimes \G) \cong H^\bullet(\mathcal A_\theta^{alg} \rtimes \G, (\mathcal A_\theta^{alg} \rtimes \G)^\ast).$$
$$HP_\bullet(\mathcal A_\theta^{alg} \rtimes \G) \cong HP^\bullet(\mathcal A_\theta^{alg} \rtimes \G).$$
We conjecture that this duality will hold for the noncommutative smooth orbifold $\mathcal A_\theta \rtimes \G$ under similar restriction on $\theta$ as in \cite{C}. Further we conjecture the following:
\begin{conj}
$K_0(\mathcal A_{\theta}^{alg} \rtimes \G) \cong\begin{cases}
\mathbb Z^7 & \text{ for } \G = \mathbb Z_3\\
\mathbb Z^8  & \text{ for } \G = \mathbb Z_4 \\
\mathbb Z^9 & \text{ for } \G = \mathbb Z_6. \end{cases}$\\
\end{conj}

\vspace{2mm}

{\small \noindent{Safdar Quddus},\\ School of Mathematical Sciences,\\ National Institute of Science Education and Research, Bhubaneswar, India.\\
Email: safdar@niser.ac.in.


\begin{thebibliography}{}

\bibitem[AL]{AL} J. Alev and T. Lambre: Homologie des invariants d'une alg\`ebre de Weyl, \emph{K-Theory, 18 (1999), 401--411}.

\bibitem[B]{B} J. Baudry: Invariants du tore quantique, \emph{Bull. Sci. Math., 134 (2010), 531--547}.

\bibitem[BRT]{BRT} Y. Berest, A. Ramadoss and X. Tang: The Picard group of a noncommutative algebraic torus,  \emph{J. Noncommut. Geom., 7 (2013), 335--356.}.
\bibitem[BW]{BW}  J. Buck and S. Walters: Non commutative spheres associated with the hexic transform and their K-theory, \emph{J. Operator Theory
58:2(2007), 441 -- 462}.

\bibitem[C]{C} A. Connes: Noncommutative differential geometry, \emph{IHES Publ.
Math., 62 (1985), 257--360}.

\bibitem[ELPH]{ELPH} S. Echterhoff, W. L\"uck, N. Phillips and S. Walters: The structure of
crossed products of irrational rotation algebras by finite subgroups of $
SL_2(\mathbb{Z})$, \emph{J. Reine Angew. Math., 639 (2010), 173--221}.

\bibitem[EO]{EO} P. Etingof and A. Oblomkov: Quantization, orbifold cohomology, and Cherednik algebras, Jack, Hall-Littlewood and Macdonald polynomials, 
\emph{Contemp. Math., 417, Amer. Math. Soc., Providence, RI, (2006),  171--182.} 

\bibitem[F]{F} S. Fryer: The q-Division Ring and its Fixed Rings, \emph{Journal of Algebra, Volume 402, pp. 358-378 (2013)}.

\bibitem[GJ]{GJ} E. Getzler and J.D.S. Jones: The cyclic homology of crossed product
algebras, \emph{J. Reine Angew. Math., 445 (1993), 161--174}.

\bibitem[HT]{HT} G. Halbout and X. Tang: Noncommutative Poisson structures on orbifolds, \emph{Trans. Amer. Math. Soc., 362 (2010), 2249--2277}.

\bibitem[L]{L} Loday, J: : Cyclic homology(ISBN 3540630740) \emph{Springer, Second Edition}.

\bibitem[NPPT]{NPPT} N. Neumaier, M.J. Pflaum, H.B. Posthuma and X. Tang: Homology of formal deformations of proper \'etale Lie groupoids,
\emph{J. Reine Angew. Math., 593 (2006), 117--168}.


\bibitem[O]{O} A. Oblomkov: Double affine Hecke algebras of rank 1 and affine cubic surfaces, \emph{Int. Math. Res. Not., no. 18 (2004), 877--912.} 

\bibitem[PV]{PV} M. Pimsner and D. Voiculescu: Imbedding the irrational rotation
C*-algebra into an AF-algebra, \emph{J. Operator Theory, 4 (1980), 201--210}.

\bibitem[Q1]{Q1} S. Quddus: Hochschild and cyclic homology of the crossed product of algebraic irrational rotational algebra by finite subgraoups 
of $SL(2,\mathbb Z)$, \emph{J. Algebra 447 (2016), 322--366}.

\bibitem[Q2]{Q2} S. Quddus: Cohomology of $\mathcal A_\theta^{alg} \rtimes \mathbb Z_2$ and its Chern-Connes Pairing, \emph{Journal of Noncommutative Geometry, accepted}

\bibitem[Y]{Y} A. Yashinski: The Gauss-Manin connection for the cyclic homology of smooth deformations, and noncommutative tori, \emph{Journal of Noncommutative Geometry, to appear}.

\end{thebibliography}
\end{document}